\documentclass[11pt]{amsart}  
\usepackage{amsthm, amsmath, amscd, amssymb, latexsym, stmaryrd, color}
\usepackage[top=3.0cm, bottom=3.0cm, left=2.8cm, right=2.8cm]{geometry}
\theoremstyle{plain}

\newtheorem*{thm1}{Theorem 1}
\newtheorem*{thm2}{Theorem 2}

\newtheorem{theorem}{Theorem}[section]
\newtheorem{lemma}[theorem]{Lemma}
\newtheorem{proposition}[theorem]{Proposition}

\theoremstyle{definition}

\newtheorem{condition}[theorem]{Condition}
\newtheorem*{condition-0}{Condition-$\mathscr B_0$}
\newtheorem*{condition-b}{Condition-$\mathscr B$}
\newtheorem{definition}[theorem]{Definition}

\newtheorem{notation}[theorem]{Notation}
\newtheorem{example}[theorem]{Example}

\newtheorem{remark}[theorem]{Remark}
\theoremstyle{remark}
\newtheorem*{ack}{Acknowledgement}
\usepackage[mathscr]{eucal}
\usepackage{graphics, graphpap}
\usepackage{array, tabularx, longtable}
\usepackage{color}
\numberwithin{equation}{section}

\def\gcd{\mathrm{gcd}}

\def\root{\mathrm{root}}

\title[$p$-adic zeta functions of plane curve singularities]{The $p$-adic zeta function of a plane curve singularity} 
\author[H.T. Ho\`ang]{Huyen Trang Ho\`ang}
\address{VNU University of Science, Vietnam National University, Hanoi \newline \indent 
	334 Nguyen Trai Street, Thanh Xuan District, Hanoi, Vietnam}
\email{hh.trang2409@gmail.com}
\thanks{}

\author[Q.T. L\^e]{Quy Thuong L\^e}
\address{VNU University of Science, Vietnam National University, Hanoi \newline \indent 
	334 Nguyen Trai Street, Thanh Xuan District, Hanoi, Vietnam}
\email{leqthuong@gmail.com}
\thanks{}

\author[H.L. Nguyen]{Hoang Long Nguyen}
\address{VNU University of Science, Vietnam National University, Hanoi \newline \indent 
334 Nguyen Trai Street, Thanh Xuan District, Hanoi, Vietnam}
\email{nguyenhoanglong\_t66@hus.edu.vn}
\thanks{}

\keywords{$p$-adic fields, $p$-adic toric modifications, $p$-adic zeta function, plane curve singularity, Newton polyhedra}
\subjclass[2020]{Primary 14B05, 14H20, 11S80, 11S40; Secondary 52B05}

\begin{document}           
\begin{abstract}
Using toric modifications and some compatibility we compute the local $p$-adic zeta function of a plane curve singularity. Thanks to the compatibility, we can work over the analytic change of variables formula for $p$-adic integrals, hence avoid adapting to the algebraic setting and Denef's formula.   
\end{abstract}
\maketitle                 

\section{Introduction}\label{sec1}
\subsection{Plane curve singularities} 
Up to now, there have been many excellent results in the study of hypersurface singularities in clasical and non-Archimedean geometries and arithmetic. In singularity theory, a surprising bridge connecting the geometry and arithmetic is the monodromy conjecture, where the number of solutions of an integer-coefficient polynomial equation modulo powers of a prime would be modeled on the singularities of the complex hypersurface defined by the vanishing of the polynomial. For plane curves, Loeser \cite{Loe88} and Rodrigues \cite{Rod04} respectively proved the conjecture associated to the Bernstein polynomial and the topological zeta function, while N\'emethi and Veys \cite{NV10} gave a generalized result. 

Furthermore, for plane curve singularities, Guibert \cite{G02} computed the Milnor motivic fiber in terms of the invariants of the value semigroup, Melle-Hernandez, Torrelli and Veys \cite{MHTV10} studied monodromy Jordan blocks, $b$-functions and poles of zeta functions,  L\^e \cite{Thuong1, Thuong2} described the monodromy zeta function and the Milnor motivic fiber. L\^e-Nguyen \cite{Thuong-Hung} described the topological zeta function and gave a new proof of the monodromy conjecture for curves using toric modifications. More recently, Cassou-Noguès and Raibaut \cite{CNR18, CNR21} studied the Milnor motivic fiber, the Milnor motivic fiber at infinity and the Milnor motivic fiber of rational functions.

In this paper, we are going to compute the $p$-adic zeta function of a plane curve singularity using the methods in \cite{AO} as well as in our previous works \cite{Thuong1, Thuong2, Thuong-Hung}. Since $p$-adic fields are not algebraically closed, additional assumptions need to be used to make the computation feasible.

\subsection{$p$-adic zeta functions}
Let $p$ be a prime natural number, and let $(K,|\cdot |)$ be a $p$-adic field with the ring of $p$-adic integral numbers 
$$\mathcal O_K=\{\xi\in K\mid |\xi|\leq 1\},$$ 
the maximal ideal 
$$\mathcal M_K=\{\xi\in K\mid |\xi|<1\},$$ 
and the residue field $\mathbb F_q=\mathcal O_K/\mathcal M_K$, where $q$ is a power of $p$. Let $v$ denote the $\mathcal M_K$-valuation, so that $|\xi|=q^{-v(\xi)}$ for $\xi\in K^*$ and $v(0)=\infty$. Fix a generator $\pi$ of $\mathcal M_K$. It is a fact that $K$ is a topological field and a basis of the topology at the origin is formed from the sets $\pi^i\mathcal O_K$ for integers $i$. Because $\pi^i\mathcal O_K$'s are compact open subgroups of $(K,+)$ and $K=\bigcup_{i\in \mathbb Z}\pi^i\mathcal O_K$, $K$ is a locally compact totally disconnected group, thus for any integer $n\geq 1$, $(K^n,+)$ admits a unique canonical Haar measure $|dx|$ such that $|dx|(\mathcal O_K^n)=1$.

Let $f(x)$ be a polynomial over $\mathcal O_K$ in $n$ variables $x=(x_1,\dots,x_n)$. We are interested in the simplest form of $p$-adic zeta functions of $f$ defined as follows
$$Z(f;s):=\int_{\mathcal O_K^n}|f(x)|^s|dx|.$$
It is straightforward that $Z(f;s)$ is well defined on the complex domain given by $\Re(s)>0$. Igusa, however, showed that $Z(f;s)$ can be extended to a meromorphic function on the whole complex plane; he also proved that it is a rational function in $q^{-s}$. By the $p$-adic monodromy conjecture, for $f$ over $\mathbb Z$ and $p$ sufficiently large, it is expected that the poles of $Z(f;s)$ provide the eigenvalues of the monodromy at complex points of $f^{-1}(0)\subseteq \mathbb C^n$. If we denote for each natural number $m$ by $N_m$ the number of solutions in $(\mathcal O_K/(\pi^m))^n$ of the congruence equation $f(x)\equiv 0\mod \pi^m$, we obtain the Poincar\'e series 
$$P_f(t)=\sum_{m\geq 0}N_m(q^{-n}t)^m,$$ 
which relates to $Z(f;s)$ by the identification 
$$(1-q^{-s})P_f(q^{-s})=1-q^{-s}Z(f;s).$$

Not only for polynomials, one can also define the $p$-adic integral and the local zeta function for $K$-analytic functions, i.e. for formal power series over $K$ convergent in a subset of $K^n$. It is a fact that any series $f(x)=\sum_{\alpha\in \mathbb N^n}a_{\alpha}x^{\alpha}$ in $K[[x]]$ has its radius of convergence to be 
$$\rho(f)=\frac{1}{\limsup_{|\alpha|\to\infty}|a_{\alpha}|^{1/|\alpha|}}.$$ 
In particular, if $f(x)$ is in $\mathcal O_K[[x]]$ it is a $K$-analytic function on $\mathcal M_K^{\oplus n}$. Now, assume in addition that $f(x)$ vanishes at the origin $\mathbf 0$ of $K^n$. Consider the $p$-adic local zeta function of $f$ at $\mathbf 0$ defined as follows
$$Z_0(f;s):=\int_{\mathcal M_K^{\oplus n}}|f(x)|^s|dx|.$$

We will focus on the two cases mentioned previously. Let $\Phi: Y\to K^n$ be a $K$-analytic morphism of compact
$K$-analytic manifolds such that $\Phi$ is bi-analytic away from $(f\circ \Phi)^{-1}(0)$ (or more generally, away from any closed subsets $Z\subseteq Y$ and $\Phi(Z)\subseteq K^n$ of measure $0$). By the analytic change of variables formula (cf. \cite[Proposition 7.4.1]{Ig}) one has
\begin{align*}
\int_A|f(x)|^s|dx|=\int_{\Phi^{-1}(A)}|f\circ \Phi|^s|\Phi^*dx|,
\end{align*}
where $A$ is either $\mathcal O_K^n$ (when $f$ is a polynomial over $\mathcal O_K$) or $\mathcal M^{\oplus n}$, from which $Z(f;s)$ and $Z_0(f;s)$ can be expressed in terms of the numerical data of $\Phi$. Since Denef's formula for $p$-adic zeta functions using good reduction resolution of singularities only works in the algebraic setting, while our context is mainly analytic, we are going to apply the analytic change of variables formula to toric modifications in a crucial way in this paper.


\subsection{Main results}
Fix an algebraic closure $\overline{K}$ of $K$, and suppose $x$ (resp. $y$) to be a single variable. Let $f(x,y)$ be in $\mathcal O_K[[x]][y]$, which defines a plane curve $C\subseteq \overline{K}^2$ near the origin $\mathbf 0=(0,0)$. Then $f(x,y)$ admits a decomposition into irreducible components in $K[[x,y]]$, in which we can assume $f(x,y)$ equals $x^{\alpha}y^{\beta}$ times non-smooth components. Each non-smooth component $f_{i\ell\tau}(x,y)$ has its initial expansion 
$$f_{i\ell\tau}(x,y)=(y^{a_i}+\xi_{i\ell}x^{b_i})^{A_{i\ell\tau}}+\text{(higher terms)},$$
where $\xi_{i\ell}\in \overline{K}$, $\xi_{i\ell}\not=\xi_{i\ell'}$ when $\ell\not=\ell'$, $a_i\geq 2$, $b_i\geq 2$, and $A_{i\ell\tau}$ are in $\mathbb N^*$ with $(a_i,b_i)$ coprime, for $1\leq i\leq k$, $1\leq \ell \leq r_i$, and $1\leq \tau\leq s_{i\ell}$. The vectors $P_i=(a_i,b_i)^t$, $1\leq i\leq k$, correspond to the compact facets of the Newton polyhedron $\Gamma_f$ of $f$. Consider a {\it regular} simplicial cone subdivision $\Sigma$ that contains all $P_i$'s. Similarly as in \cite{AO}, $\Sigma$ gives rise to a toric modification $\Phi: X_{\Sigma}\to K^2$, which is a $K$-analytic map and a bi-analytic isomorphism outside measure-zero sets.

Assuming in addition that $\xi_{i\ell}$ is in $\mathcal O_K\setminus \mathcal M_K$ for every $1\leq i\leq k$ and $1\leq \ell\leq r_i$, the Newton polyhedra $\Gamma_f$ and $\Gamma_{\widetilde{f}}$ of $f$ and its modulo $\mathcal M_K$ reduction $\widetilde{f}$, respectively, are the same. This is possible when taking $K$ a sufficiently large extension of $\mathbb Q_p$. It guarantees that the toric modifications 
$$\Phi: X_{\Sigma}\to K^2\quad \text{and}\quad \widetilde{\Phi}: \widetilde{X}_{\Sigma}\to \mathbb F_q^2$$ 
are compatible via the reduction map. Using the compatibility we show in Proposition \ref{main-formula} how to compute $Z_0(f;s)$ in an inductive way on toric modifications. All the main results of the present paper are consequences of Proposition \ref{main-formula}, which are sketched as follows.

\begin{thm1}[Theorem \ref{main-formula-nondegenerate}, Proposition \ref{zeta-nondegenerate}]
Under the Newton nondegeneracy for $f(x,y)$ and the compatibility of $\Phi$ and $\widetilde{\Phi}$ (see more details in Theorem \ref{main-formula-nondegenerate}), $Z_0(f;s)$ is computed and equal to $\frac{q-1}{q^2}$ times
$$\sum_{j=1}^m\frac{\#\widetilde{E}(T_j)^{\circ}}{q^{N_js+\nu_j}-1}+\sum_{j=0}^m\frac{q-1}{(q^{N_js+\nu_j}-1)(q^{N_{j+1}s+\nu_{j+1}}-1)}+\sum_{i=1}^k\frac{(q-1)r_i}{(q^{N(P_i)s+\nu(P_i)}-1)(q^{s+1}-1)}$$
using numerical data $N_j$, $N(P_i)$, $\nu_j$ and $\nu(P_i)$ from $\Phi$. Furthermore, eliminating fake poles, one can find a univariate $\mathbb Q$-polynomial $\mathcal Q$ such that
$$Z_0(f;s)=\frac{\mathcal Q(q^s)}{(q^{s+1}-1)(q^{\alpha s+1}-1)(q^{\beta s+1}-1)\prod_{1\leq i\leq k}(q^{N(P_i)s+\nu(P_i)}-1)}.$$
\end{thm1}

For a general singularity $f(x,y)\in \mathcal O_K[[x]][y]$ we can construct a toric modification $\Phi_0$ associated to its Newton polyhedron $\Gamma_f$. Singularities in the strict transform of $f(x,y)$ under $\Phi_0$ require more toric modifications to resolve. There is a canonical way to create a new system of coordinates at these singular points before new toric modifications, it is the use of Tschirnhausen approximate polynomials of $f(x,y)$ introduced in \cite{AO} (not taking the factor $x^{\alpha}y^{\beta}$ into account one may assume $f(x,y)$ is a {\it monic} polynomial in $y$). Collecting all toric modifications (which is finite) we get a resolution of the singularity, hence a resolution graph $\mathbf G$. The vertices of $\mathbf G$ corresponding bijectively to a toric modification will be called a {\it bamboo}. Through its bamboos, the graph $\mathbf G$ presents the hierarchy of the toric modifications.

\begin{thm2}[Theorem \ref{zeta-general}, Theorem \ref{zeta-truepoles}]
Let $f(x,y)$ be a monic polynomial in $\mathcal O_K[[x]][y]$ such that $f(\mathbf 0)=0$. Assume that for any toric modification $\Phi$, it and its reduction $\widetilde{\Phi}$ is compatible (see more details in Theorem \ref{zeta-general}). For a top bamboo $\mathscr B$ of $\mathbf G$, put
$$Z_{\mathscr B}(s)=\frac{(q-1)^2/q^2}{(q^{N(P_{\root}^{\mathscr B})s+\nu(P_{\root}^{\mathscr B})}-1)(q^{s+1}-1)};$$
for a non-top bamboo $\mathscr B$, put 
\begin{gather*}
Z_{\mathscr B}(s)=\frac{q-1}{q^2}\sum_{j=1}^{m^{\mathscr B}}\frac{\#\widetilde{E}(T_j^{\mathscr B})^{\circ}}{q^{N_j^{\mathscr B}s+\nu_j^{\mathscr B}}-1}+\sum_{j=0}^{m^{\mathscr B}}\frac{(q-1)^2/q^2}{(q^{N_j^{\mathscr B}s+\nu_j^{\mathscr B}}-1)(q^{N_{j+1}^{\mathscr B}s+\nu_{j+1}^{\mathscr B}}-1)}\\
\qquad +\frac{(q-1)^2/q^2}{(q^{N(P_{\root}^{\mathscr B})s+\nu(P_{\root}^{\mathscr B})}-1)(q^{N_1^{\mathscr B}s+\nu_1^{\mathscr B}}-1)}.
\end{gather*}
Here $N_j^{\mathscr B}$, $\nu_j^{\mathscr B}$, $N(P_{\root}^{\mathscr B})$ and $\nu(P_{\root}^{\mathscr B})$ are given in Lemmas \ref{lem41} and \ref{lem42}. Then
$$Z_0(f;s)=\sum_{\mathscr B\in \mathbf B}Z_{\mathscr B}(s).$$
Furthermore, eliminating fake poles, one can find a univariate $\mathbb Q$-polynomial $\mathcal Q$ such that
$$Z_0(f;s)=\frac{\mathcal Q(q^s)}{(q^{s+1}-1)(q^{\alpha s+1}-1)(q^{\beta s+1}-1)\prod_{\mathscr B\in \mathbf B^{\mathrm{nt}}}\prod_{1\leq i\leq k}(q^{N(P_i^{\mathscr B})s+\nu(P_i^{\mathscr B})}-1)}.$$
\end{thm2}

Remark that the compatibility of all pairs of toric modifications $\Phi$ and $\widetilde{\Phi}$ (more precisely, Condition-$\mathscr B_0$ and Condition-$\mathscr B$) although guarantees that the strict transform $E'$ of the resolution of the singularity intersects transversally with $E(P_i^{\mathscr B})$ but it does not allow us to know the exact number of irreducible components of $E'$ (we would certainly know this number when $K$ was algebraically closed, which is impossible). This is the reason why we choose to study $Z_0(f;s)$ instead of $Z(f;s)$ when $f$ is not a polynomial.

\section{Toric modifications and the $p$-adic change of variables formula}\label{sec2}

\subsection{Toric modifications}
We consider the order on $2$-dimensional primitive integral column vectors defined as follows $P<Q$ if and only if $\det(P,Q)>0$. The primitiveness for the integral column vectors is important to guarantee that $\{<,=\}$ is an ordering relation. A {\it simplicial cone subdivision} $\Sigma$ of the first quarter $\mathbb R_{\geq 0}^2$ is a sequence $T_1<\cdots<T_m$ of primitive vectors in $\rm{Mat}(2\times 1, \mathbb N)$. The vectors $T_1, \dots,T_m$ are called {\it vertices} of $\Sigma$. Put $T_0=(1,0)^t$, $T_{m+1}=(0,1)^t$, and $\sigma_j=(T_j,T_{j+1})$, $0\leq j\leq m$. A simplicial cone subdivision $\Sigma$ is called {\it regular} if $\det\sigma_j=1$ for all $0\leq j\leq m$. 

For any integral $(2\times 2)$-matrix 
$$\sigma=\left(\begin{matrix}a&b\\ c&d \end{matrix}\right)$$ 
we consider the map 
$$\Phi_{\sigma}:K^2\rightarrow K^2$$ 
defined by 
$$\Phi_{\sigma}(x,y)=(x^ay^b,x^cy^d).$$ 
If $\det\sigma=\pm 1$, then $\Phi_{\sigma}$ is a birational map, which is inverse to $\Phi_{\sigma^{-1}}$. For a regular simplicial cone subdivision $\Sigma=\{T_1<\cdots<T_m\}$, we consider the toric charts $(K^2_{\sigma_i};x_j,y_j)$, $0\leq j\leq m$, with $K^2_{\sigma_j}$ a copy of $K^2$. We can glue these charts along $(K^*_{\sigma_j})^2$ according to the relations 
$$\Phi_{\sigma_i^{-1}\sigma_j}(x_j,y_j)=(x_i,y_i), \quad $$ 
into a $2$-dimensional $p$-adic analytic smooth manifold $X_{\Sigma}$. We define a map 
$$\Phi: X_{\Sigma}\to K^2$$ 
by putting 
$$\Phi(x_j,y_j)=\Phi_{\sigma_j}(x_j,y_j)$$ 
for $(x_j,y_j)$ in the chart $K^2_{\sigma_j}$, $0\leq j\leq m$. This map $\Phi$ is clearly well defined and it is called the {\it toric modification of $K^2$} associated to $\Sigma$. The divisor $\Phi^{-1}(\mathbf 0)$ has simple normal crossings with $m$ {\it exceptional divisors} $E(T_j)$ for $1\leq j \leq m$. Each exceptional divisor $E(T_j)$ corresponds uniquely to the vertex $T_j$ of $\Sigma$, it is contained in the union of the two consecutive charts $K^2_{\sigma_{j-1}}$ and $K^2_{\sigma_j}$, and defined on them by $y_{j-1}=0$ and $x_j=0$, respectively. Then $E(T_j)$ and $E(T_i)$ intersect if and only if $|j-i|\leq 1$, and the intersections are transversal if and only if $|j-i|=1$. The noncompact components $E(T_0)=\{x_0=0\}$ and $E(T_{m+1})=\{y_m=0\}$ are isomorphic to the coordinate axes $x=0$ and $y=0$ respectively.

The same construction as previous can be also applied to obtain the toric modification of $\mathbb F_q^2$ associated to $\Sigma$, we will denote it by $\widetilde{\Phi}: \widetilde{X}_{\Sigma}\to \mathbb F_q^2$.

\subsection{Admissible toric modifications}
Let $f(x,y)$ be in $\mathcal O_K[[x,y]]$ such that $f(\mathbf 0)=0$, and let $\Gamma_f$ (resp. $\Gamma_{\widetilde{f}}$) be the Newton polyhedron of $f$ (resp. $\widetilde{f}$) at the origin $\mathbf 0$ (resp. $\widetilde{\mathbf 0}$) of $K^2$ (resp. $\mathbb F_q^2$). Let $\widetilde{f}(x,y)$ be the reduction modulo $\mathcal M_K$ of $f(x,y)$. Fix an algebraic closure $\overline{K}$ of $K$, denote $\overline{\mathbb F}_q$ the residue field of $\overline{K}$, which is also an algebraic closure of $\mathbb F_q$. Up to the Weierstrass Preparation Theorem, we can assume that $f(x,y)$ is in $\mathcal O_K[[x]][y]$ and write a decomposition of $f(x,y)$ in $\overline{K}[[x,y]]$ as follows   
\begin{equation}\label{f-initial-expansion}
f(x,y)=cx^{\alpha}y^{\beta}\prod_{i=1}^k\prod_{\ell=1}^{r_i}\prod_{\tau=1}^{s_{i\ell}}f_{i\ell \tau}(x,y),
\end{equation}
where $c\in \mathcal O_K\setminus \mathcal M_K$, $\alpha, \beta\in\mathbb N$, and all $f_{i\ell\tau}$ are monic, non-smooth and irreducible in $\overline{K}[[x]][y]$. For the later convenience, we put $f_i=\prod_{\ell=1}^{r_i}f_{i\ell}$ and $f_{i\ell}=\prod_{\tau=1}^{s_{i\ell}}f_{i\ell \tau}$. Similarly as in \cite{AO}, each $f_{i\ell\tau}(x,y)$ admits an initial expansion in the following form
\begin{equation}\label{ff-initial-expansion}
f_{i\ell\tau}(x,y)=(y^{a_i}+\xi_{i\ell}x^{b_i})^{A_{i\ell\tau}}+\text{(higher terms)},
\end{equation}
where $\xi_{i\ell}\in \overline{K}^*$, $\xi_{i\ell}\not=\xi_{i\ell'}$ if $\ell\not=\ell'$, $a_i$, $b_i$, and $A_{i\ell\tau}$ are in $\mathbb N^*$ with $(a_i,b_i)$ coprime. Since all the components $f_{i\ell\tau}$ are non-smooth, we have $a_i\geq 2$ and $b_i\geq 2$ for $1\leq i\leq k$. Put $A_i=\sum_{\ell=1}^{r_i}A_{i\ell}$ and $A_{i\ell}=\sum_{\tau=1}^{s_{i\ell}}A_{i\ell\tau}$. 

\begin{condition}\label{assumption}
For $1\leq i\leq k$ and $1\leq \ell\leq r_i$, $\xi_{i\ell}\in \mathcal O_K\setminus \mathcal M_K$ and $f_{i\ell}$ are monic in $K[[x]][y]$.
\end{condition}

Sometimes, Condition \ref{assumption} is also called a {\it compatibility} as explained as follows. Indeed, under the condition, $\widetilde{f}$ also admits an analogous decomposition as (\ref{f-initial-expansion}) and (\ref{ff-initial-expansion}) with $f_{i\ell\tau}$ replaced by $\widetilde{f}_{i\ell\tau}$ (hence $\xi_{i\ell}$ replaced by $\widetilde{\xi}_{i\ell}$). By the previous assumptions (Condition \ref{assumption} included), the boundary of $\Gamma_f$ (resp. $\Gamma_{\widetilde{f}}$) has $k$ compact facets whose primitive normal vectors are $P_i=(a_i,b_i)^t$ with $1\leq i\leq k$. For convention, we shall denote $P_0=(1,0)^t$ and $P_{k+1}=(0,1)^t$. We assume that $P_0<P_1<\cdots<P_k<P_{k+1}$. Let $\Sigma$ be a regular simplicial cone subdivision with vertices $T_j=(c_j,d_j)^t$, $1\leq j\leq m$. It is called {\it admissible for $f$} (resp. {\it for $\widetilde{f}$}) if $\{P_1,\dots,P_k\}$ is contained in $\{T_1,\dots, T_m\}$. Consider the toric modification $\Phi: X_{\Sigma}\to K^2$ (resp. $\widetilde{\Phi}: \widetilde{X}_{\Sigma}\to \mathbb F_q^2$) associated to $\Sigma$. Then, $\Phi$ (resp. $\widetilde{\Phi}$) is called {\it admissible for $f$} (resp. {\it for $\widetilde{f}$}) if $\Sigma$ is admissible for $f$ (resp. for $\widetilde{f}$). The exceptional divisors of $\Phi$ (resp. $\widetilde{\Phi}$) are $E(T_j)$ (resp. $\widetilde{E}(T_j)$) for $1\leq j\leq m$. 

Let $E'$ (resp. $\widetilde{E}'$) be the closure of $\Phi^{-1}(f^{-1}(0)\setminus\{\mathbf 0\})$ (resp. $\widetilde{\Phi}^{-1}(\widetilde{f}^{-1}(\widetilde{0})\setminus\{\widetilde{\mathbf 0}\})$) in $X_{\Sigma}$ (resp. $\widetilde{X}_{\Sigma}$). We call $E'$  (resp. $\widetilde{E}'$) the {\it strict transform} of $f$ (resp. $\widetilde{f}$) under $\Phi$ (resp. $\widetilde{\Phi}$).

\medskip 
In the rest of this subsection, we fix a toric modification $\Phi: X_{\Sigma}\to K^2$ admissible for $f$.

\begin{notation}
For $1\leq j\leq m$, set 
$$E(T_j)^{\circ}:=E(T_j)\setminus \big(\bigcup_{j'\not=j}E(T_{j'})\cup E'\big).$$ 
We write $N_j$ or $N(T_j)$ (resp. $\nu_j-1$ or $\nu(T_j)-1$) for the multiplicity on $E(T_j)$ of $\Phi^*f$ (resp. $\Phi^*(dx\wedge dy)$) for $0\leq j\leq m+1$. By observation we have $N(T_0)=\alpha$, $N(T_{m+1})=\beta$. For $0\leq i\leq k$, let $j_i$ denote the index with $0\leq j_i \leq m+1$ and $T_{j_i}=P_i$. 
\end{notation}

\begin{lemma}\label{local-form}
Let $f(x,y)$ be in $\mathcal O_K[[x,y]]$ such that $f(\mathbf 0)=0$. Assume that $f(x,y)$ admits the decomposition (\ref{f-initial-expansion}) together with (\ref{ff-initial-expansion}) and Condition \ref{assumption}. Then the following hold.
\begin{itemize}
	\item[(i)] For any $z\in E(T_j)^{\circ}$, $0\leq j\leq m+1$, there exist a local system of coordinates $(u,v)$ at $z$ and units $U_j$, $V_j$ in $K[[u,v]]$ with $|U_j|=|V_j|=1$ such that 
	\begin{align*}
	\Phi^*f(u,v)&=u^{N_j}U_j(u,v),\\ 
	\Phi^*(dx\wedge dy)&=u^{\nu_j-1}V_j(u,v)du\wedge dv.
	\end{align*}

	\item[(ii)] For $\{z\}=E(T_j)\cap E(T_{j+1})$, $0\leq j\leq m$ there is a local system of coordinates $(u,v)$ at $z$ and units $U_{j,j+1}$, $V_{j,j+1}$ in $K[[u,v]]$ with $|U_{j,j+1}|=|V_{j,j+1}|=1$ such that 
	\begin{align*}
	\Phi^*f(u,v)&=u^{N_j}v^{N_{j+1}}U_{j,j+1}(u,v),\\
	\Phi^*(dx\wedge dy)&=u^{\nu_j-1}v^{\nu_{j+1}-1}V_{j,j+1}(u,v) du\wedge dv.	
	\end{align*}
\end{itemize}
Moreover, for $0\leq j\leq m+1$,
$$\nu_j=c_j+d_j,$$ 
and for $P_i\leq T_j \leq P_{i+1}$, $0\leq i\leq k$, 
$$N_j=c_j\alpha+c_j\sum_{t=1}^ir_tb_t+d_j\sum_{t=i+1}^kr_ta_t+d_j\beta.$$
\end{lemma}

\begin{proof}
(i) The cases $j=0$ and $j=m+1$ are trivial, let us consider the remaining cases. We work with the chart $(K_{\sigma_j}^2;x_j,y_j)$. First, it is straightforward that $\Phi^*(dx\wedge dy)=x_j^{\nu_j-1}y_j^{\nu_{j+1}-1}dx_j\wedge dy_j$. For $P_t< T_j$, resp. $T_j < P_t$, we have 
\begin{align*}
\Phi^*f_t(x_j,y_j)&=x_j^{b_tc_jA_t}y_j^{b_tc_{j+1}A_t}\Big(\prod_{\ell=1}^{r_t}\xi_{t\ell}^{A_{t\ell}}+x_jR_t\Big),\\
\text{resp.} \quad \Phi^*f_t(x_j,y_j)&=x_j^{a_td_jA_t}y_j^{a_td_{j+1}A_t}\big(1+x_jR_t\big),
\end{align*}
for some $R_t\in \overline{K}[[x_j,y_j]]$. For $T_j=P_i$, we have, on $(K_{\sigma_j}^2;x_j,y_j)$, $$\Phi^*f_i(x_j,y_j)=x_j^{a_ib_iA_i}y_j^{b_ic_{i+1}A_i}\left(\prod_{\ell=1}^{r_i}(\xi_{i\ell}+y_j)^{A_{i\ell}}+x_jR_i\right),$$ 
for some $R_i\in \overline{K}[[x_j,y_j]]$. By Condition \ref{assumption}, the strict transform $E$ of $f$ intersects with $E(P_i)$ at the points $(0,-\xi_{i\ell})$, $1\leq \ell\leq r_i$, in the chart $(K_{\sigma_j}^2;x_j,y_j)$. Thus
$$\Phi^*f(x_j,y_j)=x_j^{N_j}y_j^{N_{j+1}}\Big(\prod_{P_t\leq T_j}\prod_{\ell=1}^{r_t}\xi_{t\ell}^{A_{t\ell}}+\text{\rm(higher terms)}\Big),$$
with $N_j$ as in the lemma. For $z=(0,\lambda)\in E(T_j)^{\circ}$ with $\widetilde{\lambda}\not= 0$ in $\mathbb F_q$, putting $u=x_j$ and $v=-\lambda+y_j$ we have 
$$\Phi^*f=u^{N_j}U_j(u,v)$$
and
$$\Phi^*(dx\wedge dy)=u^{\nu_j-1}V_j(u,v) du\wedge dv,$$ 
where 
$$U_j(u,v)=\lambda^{N_{j+1}}\prod_{P_t\leq T_j}\prod_{\ell=1}^{r_t}\xi_{t\ell}^{A_{t\ell}}+\text{\rm(higher terms)}$$
and
$$V_j(u,v)=(\lambda+v)^{\nu_{j+1}-1},$$
which are formal series in $K[[u,v]]$. Again by Condition \ref{assumption}, $\widetilde{\lambda}^{N_{j+1}}\prod_{P_t\leq T_j}\prod_{\ell=1}^{r_t}\widetilde{\xi}_{t\ell}^{A_{t\ell}}$ is a nonzero element in $\mathbb F_q$. Since $\widetilde{\lambda}^{\nu_{j+1}-1}$ is also nonzero in $\mathbb F_q$, it follows that $|U_j(u,v)|=|V_j(u,v)|=1$ for all $(u,v)$ with $|u|<1$ and $|v|<1$.

(ii) The proof uses the same arguments as that of (i), in which we put $u=x_j$ and $v=y_j$.
\end{proof}

Before going further, let us recall the definition of Tschirnhausen  approximate polynomials of a monic polynomial, mentioned in \cite[Section 2]{AO}.

\begin{definition}
Let $g(y)=y^n+\sum_{i=1}^nc_iy^{n-i}$ be a monic polynomial in $\mathcal O_K[[x]][y]$, and let $a$ be in $\mathbb N^*$ dividing $n$. The {\it $n/a$-th Tschirnhausen  approximate polynomial} of $g$ is the monic polynomial $h(y)\in \mathcal O_K[[x]][y]$ of degree $a$ such that $\deg(f(y)-h(y)^{n/a})<n-a$.
\end{definition} 

Let us fix an $i_0$ and an $\ell_0$ with $1\leq i_0\leq m$ and $1\leq \ell_0\leq r_i$. By \cite[Section 4.3]{AO}, the $A_{i_0\ell_0}$-th Tschirnhausen approximate polynomial of $f_{i_0\ell_0}$ has the form
$$h_{i_0\ell_0}=y^{a_{i_0}}+\xi_{i_0\ell_0}x^{b_{i_0}}+\text{(higher terms)}.$$ 
Assume that $T_j=P_{i_0}$. As in the proof of Lemma \ref{local-form}, on the chart $(K_{\sigma_j}^2;x_j,y_j)$ we have
$$\Phi^*f(x_j,y_j)=cx_j^{N(P_i)}y_j^{N(T_{j+1})}\left((y_j+\xi_{i_0\ell_0})^{A_{i_0\ell_0}}+x_jR(x_j,y_j)\right)$$
and
$$\Phi^*h_{i_0\ell_0}(x_j,y_j)=x_j^{a_{i_0}b_{i_0}}y_j^{c_{j+1}b_{i_0}}(y_j+\xi_{i_0\ell_0}+x_jR'(x_j,y_j)),$$
where $c=\prod_{t<i_0}\prod_{\ell=1}^{r_t}\xi_{t\ell}^{A_{t\ell}}$, which is in $\mathcal O_K\setminus\mathcal M_K$ because of Condition \ref{assumption}. Also by Condition \ref{assumption}, $R(x_j,y_j)$ and $R'(x_j,y_j)$ are in $K[[x_j,y_j]]$. Similarly as in \cite{AO}, there is a canonical change of variables using the Tschirnhausen approximate polynomial $h_{i\ell}$ as follows
\begin{equation*}
\begin{cases}
u=x_j\\
v=y_j^{c_{j+1}b_{i_0}}(y_j+\xi_{i_0\ell_0}+x_jR'(x_j,y_j)).
\end{cases}
\end{equation*}
The local coordinates $(u,v)$ is called the {\it Tschirnhausen coordinates} at the intersection point of $E(P_{i_0})$ with $E'$ associated to $\ell_0$. By \cite[Proposition 7.4.1]{Ig}, $|du\wedge dv|=|dx_j\wedge dy_j|$. Since $\xi_{i\ell}\not=0$, in the Tschirnhausen coordinates $(u,v)$ the pullback $\Phi^*f$ has the following initial expansion
\begin{align}\label{Tschirnhausen}
f'[i_0,\ell_0](u,v):=\Phi^*f(u,v)=c'u^{N(P_{i_0})}\prod_{i=1}^{k'}\prod_{\ell=1}^{r'_i}f'_{i\ell}(u,v),
\end{align}
where $c'\in \mathcal O_K\setminus \mathcal M_K$, $f'_{\i\ell}(u,v)=\prod_{\tau=1}^{s'_{i\ell}}f'_{i\ell\tau}(u,v)$, and $f'_{i\ell\tau}(u,v)$ are monic and irreducible in $\overline{K}[[u]][v]$. The irreducible components $f'_{i\ell\tau}(u,v)$ have the following initial expansion
\begin{align*}\label{Tschirnhausen-irred} f'_{i\ell\tau}(u,v)=(v^{a'_i}+\xi'_{i\ell}u^{b'_i})^{A'_{i\ell\tau}}+\text{(higher terms)}
\end{align*}
in $\overline K[[u,v]]$, where $\xi'_{i\ell}$'s are in $\overline{K}^*$ and piecewise distinct. Remark that $a'_i$ for $1\leq i\leq k'$ are invariants, so if we use the $A_{i_0\ell_0\tau}$-th Tschirnhausen approximate polynomial of $f_{i_0\ell_0\tau}$ and express $\Phi^*f_{i_0\ell_0\tau}$ in the corresponding Tschirnhausen coordinates, we also get these numbers $a'_i$. But the latter situation is exactly the one in \cite[Section 4.3]{AO} with complex context, hence we can apply the method in \cite{AO} to show that $a'_i \geq 2$ for all $1\leq i\leq k'$. 

\begin{proposition}\label{main-formula}
Let $f(x,y)$ be in $\mathcal O_K[[x,y]]$ such that $f(\mathbf 0)=0$. Assume that $f(x,y)$ admits the decomposition (\ref{f-initial-expansion}) and (\ref{ff-initial-expansion}). Assume in addition that Condition \ref{assumption} is satisfied. Then $Z_0(f;s)$ is equal to 
$$\frac{q-1}{q^2}\left(\sum_{j=1}^m\frac{\#\widetilde{E}(T_j)^{\circ}}{q^{N_js+\nu_j}-1}+\sum_{j=0}^m\frac{q-1}{(q^{N_js+\nu_j}-1)(q^{N_{j+1}s+\nu_{j+1}}-1)}\right)
$$
plus
$$\sum_{i=1}^k\sum_{\ell=1}^{r_i}\int_{\mathcal M_K^{\oplus 2}}|f'[i,\ell](u,v)|^s|u^{\nu(P_i)-1}du\wedge dv|,$$
with the convention for the second sum that	the term corresponding to $j=0$ (resp. $j=m$) is $0$ if $\alpha=0$ (resp. $\beta=0$).
\end{proposition}

\begin{proof}
By Condition \ref{assumption}, the two toric modifications $\Phi$ and $\widetilde{\Phi}$ associated to the same simplicial cone subdivision $\Sigma$ are compatible. More precisely, we have the following commutative diagram
$$
\begin{CD}
X_{\Sigma}(\mathcal O_K) @>\Phi>> \mathcal O_K^2\\
@V\pi_1VV  @VV\pi_1V\\
\widetilde{X}_{\Sigma} @>\widetilde{\Phi}>> \ (\mathbb F_q)^2,
\end{CD}
$$	
where the upper horizontal arrow is the restriction of the toric modification $\Phi$, and the vertical ones are induced by the reduction map $\mathcal O_K\to \mathbb F_q$. From the diagram we have  
$$\pi_1^{-1}(\widetilde{\Phi}^{-1}(\widetilde{\mathbf 0}))=\Phi^{-1}(\pi_1^{-1}(\widetilde{\mathbf 0}))=\Phi^{-1}(\mathcal M_K).$$ 
By the change of variables formula, 
\begin{align*}
Z_0(f;s)&=\int_{\Phi^{-1}(\mathcal M_K^{\oplus 2})}|\Phi^*f|^s|\Phi^*(dx\wedge dy)|\\
&=\sum_{a\in \widetilde{\Phi}^{-1}(\widetilde{\mathbf 0})}\int_{\pi_1^{-1}(a)}|\Phi^*f|^s|\Phi^*(dx\wedge dy)|.
\end{align*}
Consider a stratification of $\widetilde{\Phi}^{-1}(\widetilde{\mathbf 0})$ as follows
\begin{equation}\label{decompPhibar}
\widetilde{\Phi}^{-1}(\widetilde{\mathbf 0})=\bigsqcup_{j=1}^m\widetilde{E}(T_j)^{\circ}\sqcup \bigsqcup_{j=0}^m\big(\widetilde{E}(T_j)\cap \widetilde{E}(T_{j+1})\big)\ \sqcup \bigsqcup_{i=1}^k\big(\widetilde{E}(P_i)\cap \widetilde{E}'\big)
\end{equation}
with the convention that $\widetilde{E}(T_0)\cap \widetilde{E}(T_1)=\emptyset$ (resp. $\widetilde{E}(T_m)\cap \widetilde{E}(T_{m+1})=\emptyset$) if $\alpha=0$ (resp. $\beta=0$). Because of (\ref{decompPhibar}), for $a\in \widetilde{\Phi}^{-1}(\widetilde{\mathbf{0}})$ we can consider the following cases.
	
{\it Case 1: $a\in \widetilde{E}(T_j)^{\circ}$.} In this case, $a$ lies in the intersection of the two charts $(\mathbb F_q)^2_{\sigma_j}$ and $(\mathbb F_q)^2_{\sigma_{j-1}}$; it has the coordinates $(0,\tilde{\lambda})$ in the chart $(\mathbb F_q)^2_{\sigma_j}$ and $(\tilde{\lambda}^{-1},0)$ in the chart $(\mathbb F_q)^2_{\sigma_{j-1}}$, for some $\tilde{\lambda}\in \mathbb F_q^*$. Lifting $\tilde{\lambda}$ to a unique element $\lambda\in \mathcal O_K\setminus \mathcal M_K$ we can easily prove that 
$$\pi_1^{-1}(a)\cap K_{\sigma_j}^2=\mathcal M_{K_{\sigma_j}}\times (\lambda+\mathcal M_{K_{\sigma_j}})\subseteq (\mathcal O_{K_{\sigma_j}})^2$$ 
and that  
$$\pi_1^{-1}(a)\cap K_{\sigma_{j-1}}^2=(\lambda^{-1}+\mathcal M_{K_{\sigma_{j-1}}})\times \mathcal M_{K_{\sigma_{j-1}}}\subseteq (\mathcal O_{K_{\sigma_{j-1}}})^2.$$ 
This equality clearly defines an $\mathcal M_K$-analytic homeomorphism between $\pi_1^{-1}(a)$ and $\mathcal M_K^{\oplus 2}$ with Jacobian having order zero. Remark that lifting $\tilde{\lambda}$ to $\lambda\in \mathcal O_K\setminus \mathcal M_K$ also means lifting $a$ to a unique point $z\in E(T_j)^{\circ}$. Under the above-mentioned homeomorphism $\pi_1^{-1}(a)\cong \mathcal M_K^{\oplus 2}$ and Lemma \ref{local-form} (i) we have 
$$\int_{\pi_1^{-1}(a)}|\Phi^*f|^s|\Phi^*(dx\wedge dy)|=\int_{\mathcal M_K^{\oplus 2}}|u^{N_js+\nu_j-1}||du\wedge dv|=\frac{q-1}{q^2}\cdot\frac{1}{q^{N_js+\nu_j}-1}.$$
	
{\it Case 2: $a\in \widetilde{E}(T_j)\cap \widetilde{E}(T_{j+1})\not=\emptyset$.} Using the same type of arguments as in Case 1, we have that there exists an $\mathcal M_K$-analytic homeomorphism between $\pi_1^{-1}(a)$ and $\mathcal M_K^{\oplus 2}$ such that its Jacobian has order zero. This together with Lemma \ref{local-form} (ii) gives us 
\begin{align*}\int_{\pi_1^{-1}(a)}|\Phi^*f|^s|\Phi^*(dx\wedge dy)|&=\int_{\mathcal M_K^{\oplus 2}}|u^{N_js+\nu_j-1}v^{N_{j+1}s+\nu_{j+1}-1}||du\wedge dv|\\
&=\frac{(q-1)^2}{q^2}\cdot\frac{1}{(q^{N_js+\nu_j}-1)(q^{N_{j+1}s+\nu_{j+1}}-1)}.
\end{align*}
	
{\it Case 3: $a=(0,-\widetilde{\xi}_{i\ell})\in \widetilde{E}(P_i)\cap \widetilde{E}'\cap (\mathbb F_q)_{\sigma_{j_i}}^2$.} Again as previous, there is an $\mathcal M_K$-analytic homeomorphism $\pi_1^{-1}(a)\cong \mathcal M_K^{\oplus 2}$ with Jacobian having order zero. Consider the Tschirnhansen coordinates $(u,v)$ at the lifting $z=(0,-\xi_{i\ell})$ of $a=(0,-\widetilde{\xi}_{i\ell})$. Using the above-mentioned homeomorphism $\pi_1^{-1}(a)\cong \mathcal M_K^{\oplus 2}$ with Jacobian having order zero, the Tschirnhansen coordinates $(u,v)$, as well as Condition \ref{assumption}, we have 
\begin{align*}\int_{\pi_1^{-1}(a)}|\Phi^*f|^s|\Phi^*(dx\wedge dy)|=\int_{\mathcal M_K^{\oplus 2}}|f'[i,\ell](u,v)|^s|u^{\nu(P_i)-1}du\wedge dv|.
\end{align*}
Summing up the previous integrals along the decomposition (\ref{decompPhibar}), except for the situations mentioned at the convention, we obtain Proposition \ref{main-formula}.	
\end{proof}

To continue applying Proposition \ref{main-formula}, it requires more condition as follows

\begin{condition}\label{assumption-2}
	For $1\leq i\leq k'$ and $1\leq \ell\leq r'_i$, $\xi'_{i\ell}\in \mathcal O_K\setminus \mathcal M_K$ and $f'_{i\ell}$ are monic in $K[[u]][v]$.
\end{condition}

Under Condition \ref{assumption-2}, the Newton polyhedron of $f'(u,v)$ (resp. of $\widetilde{f'}(u,v)$) again gives rise to a toric modification $\Phi'$ (resp. $\widetilde{\Phi'}$) associated to a regular simplicial cone subdivision $\Sigma'$ with vertices $T'_1<\cdots<T'_{m'}$ which is admissible for $f'$ (resp. for $\widetilde{f'}$).  


\section{Newton-nondegenerate singularities}\label{resolution-nondegenerate}
\subsection{The $p$-adic zeta function of a nondegenerate singularity}
Let $f(x,y)$ be in $\mathcal O_K[[x,y]]$ such that $f(\mathbf 0)=0$, and let $\Gamma_f$ be the Newton polyhedron of $f$ at the origin $\mathbf 0$ of $K^2$. As before, we again fix an algebraic closure $\overline{K}$ of $K$, denote $\overline{\mathbb F}_q$ the residue field of $\overline{K}$, which is also an algebraic closure of $\mathbb F_q$.

\begin{definition}\label{def31}
We say that $f(x,y)$ is {\it Newton-nondegenerate over $K$} if 
$$f(x,y)=cx^{\alpha}y^{\beta}\prod_{i=1}^kf_i(x,y)$$ 
in $\overline{K}[[x,y]]$, where $c\not=0$, $\alpha, \beta\in \mathbb N$, and $f_i(x,y)$ for $1\leq i\leq k$ are non-smooth irreducible components admitting the following initial expansion 
\begin{equation}\label{initial-expansion} f_i(x,y)=\prod_{\ell=1}^{r_i}(y^{a_i}+\xi_{i\ell}x^{b_i})+\text{(higher terms)},
\end{equation}
with integers $a_i\geq 2$, $b_i\geq 2$, $(a_i,b_i)=1$, $\xi_{i\ell}\not=0$, $\xi_{i\ell}\not=\xi_{i\ell'}$ whenever $\ell\not=\ell'$. 	
\end{definition}

\begin{definition}\label{def32}
We say that $f(x,y)$ is {\it Newton-nondegenerate over $\mathbb F_q$} if 
$$\widetilde{f}(x,y)=c'x^{\alpha}y^{\beta}\prod_{i=1}^kf'_i(x,y)$$ 
in $\overline{\mathbb F}_q[[x,y]]$, where $c'\not=0$, $\alpha, \beta\in \mathbb N$, and $f'_i(x,y)$ are non-smooth irreducible components admitting the following initial expansion 
\begin{equation*} f'_i(x,y)=\prod_{\ell=1}^{r_i}(y^{a_i}+\xi'_{i\ell}x^{b_i})+\text{(higher terms)},
\end{equation*}
with integers $a_i\geq 2$, $b_i\geq 2$, $(a_i,b_i)=1$, and $\xi'_{i\ell}\not=0$, $\xi'_{i\ell}\not=\xi'_{i\ell'}$ in $\overline{\mathbb F}_q$ whenever $\ell\not=\ell'$. 	
\end{definition}

According to Definition \ref{def31} (resp. Definition \ref{def32}) the boundary of $\Gamma_f$ (resp. $\Gamma_{\widetilde{f}}$) has $k$ compact facets whose primitive normal vectors are $P_i=(a_i,b_i)^t$ with $1\leq i\leq k$. Let $\Sigma$ be a regular simplicial cone subdivision with vertices $T_j=(c_j,d_j)^t$ for $1\leq j\leq m$. Consider the toric modification $\Phi: X_{\Sigma}\to K^2$ (resp. $\widetilde{\Phi}: \widetilde{X}_{\Sigma}\to \mathbb F_q^2$) associated to $\Sigma$. As $f(x,y)$ is Newton-nondegenerate over $K$ (resp. over $\mathbb F_q$), $\Phi$ (resp. $\widetilde{\Phi}$) is exactly an (embedded) resolution of singularity of $f$ (resp. $\widetilde{f}$) at $\mathbf 0$ (resp. $\widetilde{\mathbf 0}$). If Condition \ref{assumption} is satisfied, and $f(x,y)$ is Newton-nondegenerate over $K$, then $f(x,y)$ is Newton-nondegenerate over $\mathbb F_p$; moreover, $\Phi$ and $\widetilde{\Phi}$ are compatible. This property implies that $\Phi$ has good reduction modulo $\mathcal M_K$, see e.g. \cite{Gre} and \cite{De} for $K=\mathbb Q_p$. In general, Condition \ref{assumption} is satisfied if $p$ sufficiently large. 

\begin{lemma}\label{strict-transform}
Let $f(x,y)$ be in $\mathcal O_K[[x,y]]$ such that $f(\mathbf 0)=0$ and $f$ is Newton-nondegenerate over $K$. Assume that $f(x,y)$ admits the decomposition $f(x,y)=cx^{\alpha}y^{\beta}\prod_{i=1}^kf_i(x,y)$ in $\overline{K}[[x,y]]$ with $f_i(x,y)$ having the initial expansion (\ref{initial-expansion}), and that Condition \ref{assumption} is satisfied. For the toric modification $\Phi$ admissible for $f$ constructed previously, the strict transform of $f$ does not intersect with $E(T_j)$ if $T_j\not\in \{P_1,\dots,P_k\}$, it meets $E(P_i)$ at exactly $r_i$ points. 	
\end{lemma}

\begin{proof}
We study the behavior of the strict transform of $f$ in the chart $(K_{\sigma_j}^2;x_j,y_j)$. For $ P_t< T_j$, resp. $T_j < P_t$, we have 
\begin{align*}
\Phi^*f_t(x_j,y_j)&=x_j^{r_tb_tc_j}y_j^{r_tb_tc_{j+1}}\Big(\prod_{\ell=1}^{r_t}\xi_{t\ell}+x_jR_t(x_j,y_j)\Big),\\
\text{resp.} \quad \Phi^*f_t(x_j,y_j)&=x_j^{r_ta_td_j}y_j^{r_ta_td_{j+1}}\big(1+x_jR_t(x_j,y_j)\big),
\end{align*}
for some $R_t(x_j,y_j)\in \overline{K}[[x_j,y_j]]$. Thus, the equation of the strict transform in $(K_{\sigma_j}^2;x_j,y_j)$ is 
$$\prod_{P_t<T_j}\Big(\prod_{\ell=1}^{r_t}\xi_{t\ell}+x_jR_t(x_j,y_j)\Big)\cdot \prod_{T_j<P_t}\big(1+x_jR_t(x_j,y_j)\big)=0,$$
where the left hand side is actually an element of $K[[x_j,y_j]]$, hence the first part of the lemma is proved.

When $T_j=P_i$ for some $i$, we have, on $(K_{\sigma_j}^2;x_j,y_j)$, $$\Phi^*f_i(x_j,y_j)=x_j^{r_ia_ib_i}y_j^{r_ib_ic_{i+1}}\left(\prod_{\ell=1}^{r_i}(\xi_{i\ell}+y_j)+x_jR_i(x_j,y_j)\right),$$ 
for some $R_i(x_j,y_j)\in \overline{K}[[x_j,y_j]]$. By the hypothesis (\ref{assumption}), the strict transform of $f$ intersects transversally with $E(P_i)$ at the points $(0,-\xi_{i\ell})$, $1\leq \ell\leq r_i$, in the chart $(K_{\sigma_j}^2;x_j,y_j)$. 
\end{proof}

\begin{theorem}\label{main-formula-nondegenerate}
Assume that $f(x,y)$ and $\Phi$ are given as in Lemma \ref{strict-transform}, and that Condition \ref{assumption} is satisfied. Then $Z_0(f;s)$ is equal to $\frac{q-1}{q^2}$ times
$$\sum_{j=1}^m\frac{\#\widetilde{E}(T_j)^{\circ}}{q^{N_js+\nu_j}-1}+\sum_{j=0}^m\frac{q-1}{(q^{N_js+\nu_j}-1)(q^{N_{j+1}s+\nu_{j+1}}-1)}+\sum_{i=1}^k\frac{(q-1)r_i}{(q^{N(P_i)s+\nu(P_i)}-1)(q^{s+1}-1)},$$
with the convention for the second sum that	the term corresponding to $j=0$ (resp. $j=m$) is $0$ if $\alpha=0$ (resp. $\beta=0$).
\end{theorem}

\begin{proof}
This is an immediate consequence of Proposition \ref{main-formula} and Lemma \ref{strict-transform}. Indeed, to get the formula we note that in this situation 
$$f'[i\ell](u,v)=c'u^{N(P_i)}vU(u,v)$$ 
in the Tschirnhausen $(u,v)$ at each point of $E(P_i)\cap E'$, with $|U(u,v)|=1$, and we also note that $\#\widetilde{E}(P_i)\cap \widetilde{E}'=r_i$ and $\#\widetilde{E}(T_j)\cap \widetilde{E}(T_{j+1})=1$ (except for the situations mentioned at the convention).	
\end{proof}

\begin{definition}
Let $\pi$ be an embedded resolution of singularity of a plane curve singularity $f(x,y)$ over a field of characteristic zero with the set $\{E_j\}_{j\in J}$ of irreducible components of the exceptional divisor and the strict transform. Then the {\it extended resolution graph} $\mathbf G=\mathbf G(f,\pi)$ is the graph in which the vertices correspond to $\{E_j\}_{j\in J}$ and two vertices $E_i$ and $E_j$ are connected by a single edge only when $E_i\cap E_j$ is nonempty.	
\end{definition}

In our situation (i.e. with hypotheses as in Lemma \ref{strict-transform}), $\Phi$ is an embedded resolution of singularity. We observe that $E(T_0)$ (resp. $E(T_{m+1})$) is a vertex of this graph if and only if $\alpha>0$ (resp. $\beta>0$). The below pictures illustrate extended resolution graphs of two Newton-nondegenerate singularities corresponding to ($\alpha=\beta=0$) and ($\alpha>0$, $\beta=0$), respectively, for admissible toric modifications $X_{\Sigma}\to K^2$, in which $r_i$ arrows starting from $E(P_i)$ may or may not meet one another. 

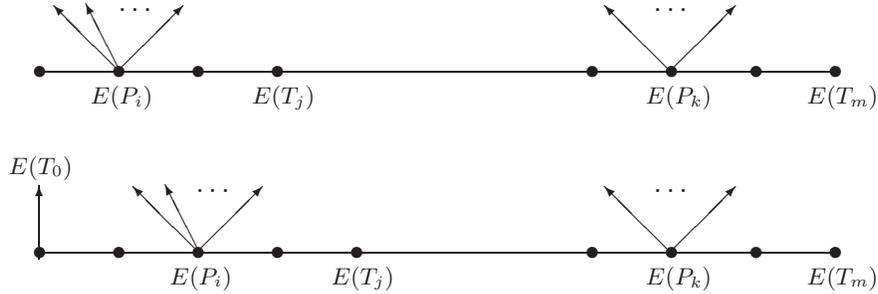
\begin{figure}[ht]
	\centering
	\begin{minipage}{3.1in}
		\begin{picture}(220,50)(-90,-15)
		\put(-130,0){\line(1,0){302}}
		\put(-109,-12){\footnotesize{$E(P_i)$}}
		\put(-99,0){\vector(1,1){25}}
		\put(-98,0){\vector(-1,1){25}}
		\put(-97.8,0){\vector(-1,2){13}}
		\put(-98.6,23){$\dots$}
		
		\put(-48,-12){\footnotesize{$E(T_j)$}}
		\put(101,-12){\footnotesize{$E(P_k)$}}
		\put(110,0){\vector(1,1){25}}
		\put(110.5,0){\vector(-1,1){25}}
		\put(104,23){$\dots$}
		
		\put(162,-12){\footnotesize{$E(T_m)$}}
		
		\put(-131,-2.6){$\bullet$} 
		\put(-101,-2.6){$\bullet$}
		\put(-71,-2.6){$\bullet$}
		\put(-41,-2.6){$\bullet$}
		
		\put(170,-2.6){$\bullet$}
		\put(140,-2.6){$\bullet$}
		\put(108,-2.6){$\bullet$}
		\put(78,-2.6){$\bullet$}
		\end{picture}
	\end{minipage}
\end{figure}

\begin{figure}[ht]
	\centering
	\begin{minipage}{3.1in}
		\begin{picture}(220,50)(-90,-20)
		\put(-130,0){\line(1,0){302}}
		\put(-79,-12){\footnotesize{$E(P_i)$}}
		\put(-128.5,-2.6){\vector(0,1){28}} 
		\put(-140,30){\footnotesize{$E(T_0)$}}
		\put(-69,0){\vector(1,1){25}}
		\put(-68,0){\vector(-1,1){25}}
		\put(-67.8,0){\vector(-1,2){13}}
		\put(-69.2,23){$\dots$}
		
		\put(-18,-12){\footnotesize{$E(T_j)$}}
		\put(101,-12){\footnotesize{$E(P_k)$}}
		\put(110,0){\vector(1,1){25}}
		\put(110.5,0){\vector(-1,1){25}}
		\put(104,23){$\dots$}
		
		\put(162,-12){\footnotesize{$E(T_m)$}}
		
		\put(-131,-2.6){$\bullet$} 
		\put(-101,-2.6){$\bullet$}
		\put(-71,-2.6){$\bullet$}
		\put(-41,-2.6){$\bullet$}
		\put(-11,-2.6){$\bullet$}
		
		\put(170,-2.6){$\bullet$}
		\put(140,-2.6){$\bullet$}
		\put(108,-2.6){$\bullet$}
		\put(78,-2.6){$\bullet$}
		\end{picture}
	\end{minipage}
\end{figure}

\subsection{Elimination of fake poles of $Z_0(f;s)$}
Let $f(x,y)$ be in $\mathcal O_K[[x,y]]$ such that $f(\mathbf 0)=0$ and $f(x,y)$ is Newton-nondegenerate over $K$, which has the initial expansion as in the previous subsection. Assume that Condition \ref{assumption} is satisfied and that we are given a toric modification $\Phi$ admissible for $f$ as above. 

\begin{proposition}\label{zeta-nondegenerate}
Assume that $f(x,y)$ and $\Phi$ are given exactly as in Lemma \ref{strict-transform}, and that Condition \ref{assumption} is satisfied. Then 
$$Z_0(f;s)=\frac{\mathcal P(q^s)}{(q^{\alpha s+1}-1)(q^{\beta s+1}-1)(q^{N(P_1)s+\nu(P_1)}-1)(q^{N(P_k)s+\nu(P_k)}-1)}+ \sum_{i=1}^{k-1}Z_i +\widetilde Z,$$ 
where 
$$Z_i=\frac{\mathcal P_i(q^s)}{\big(q^{N(P_i)s+\nu(P_i)}-1\big)\big(q^{N(P_{i+1})s+\nu(P_{i+1})}-1\big)},\quad 1\leq i\leq k,$$
for some univariate $\mathbb Q$-polynomials $\mathcal P$, $\mathcal P_i$, $1\leq i< k$, and  
\begin{gather*}
\widetilde Z=\frac{q-1}{q^2}\sum_{i=1}^k\left(\frac{q-r_i-1}{q^{N(P_i)s+\nu(P_i)}-1}+\frac{(q-1)r_i}{\big(q^{N(P_i)s+\nu(P_i)}-1\big)\big(q^{s+1}-1\big)}\right).
\end{gather*}
In other words, there is a univariate $\mathbb Q$-polynomial $\mathcal Q$ such that
$$Z_0(f;s)=\frac{\mathcal Q(q^s)}{(q^{s+1}-1)(q^{\alpha s+1}-1)(q^{\beta s+1}-1)\prod_{1\leq i\leq k}(q^{N(P_i)s+\nu(P_i)}-1)}.$$
\end{proposition}

\begin{proof}
We use the toric modification $\Phi$ admissible for $f$ as previous and apply Proposition \ref{main-formula-nondegenerate} to $\Phi$. First, we calculate the Cardinality of $\widetilde{E}(T_j)^{\circ}$ for $0\leq j\leq m+1$ in the following table.

\begin{center}
\renewcommand{\arraystretch}{1.5}
\begin{tabular}{|c|c|c|}
\hline 
{\quad } & {Cardinality} & {Conditions} \\ 
\hline
$\widetilde{E}(T_0)^{\circ}$, $\widetilde{E}(T_1)^{\circ}$ & $q-1$ & $\alpha>0$ \\
\hline
$\widetilde{E}(T_1)^{\circ}$ & $q$ & $\alpha=0$ \\
\hline
$\widetilde{E}(T_j)^{\circ}$ & $q-1$ & $1<j< m,\ T_j\not=P_i\ (\forall\ i)$\\
\hline
$\widetilde{E}(T_m)^{\circ}$ & $q$ & $\beta=0$ \\
\hline
$\widetilde{E}(T_m)^{\circ}$, $\widetilde{E}(T_{m+1})^{\circ}$ & $q-1$ & $\beta>0$ \\
\hline
$\widetilde{E}(P_i)^{\circ}$ & $q-r_i-1$ & $1\leq i\leq k$\\
\hline 
\end{tabular}
\end{center}
\medskip 
Thanks to this table, it follows from Proposition \ref{main-formula-nondegenerate} that 
$$Z_0(f;s)= \frac{(q-1)^2}{q^2}\sum_{i=1}^{k-1}S_i + \frac{q-1}{q^2}(\widetilde S_0+\widetilde S_k)+\widetilde Z,$$ 
where $\widetilde Z$ is given in the theorem, while $S_i$, $\widetilde{S}_0$ and $\widetilde{S}_k$ are defined as follows. For $1\leq i<k$,
\begin{align}\label{sum-Si}
S_i=\sum_{j=j_i}^{j_{i+1}-1}\frac{1}{(q^{N_js+\nu_j}-1)(q^{N_{j+1}s+\nu_{j+1}}-1)}+\sum_{j=j_i+1}^{j_{i+1}-1}\frac{1}{q^{N_js+\nu_j}-1}.
\end{align}
For $\alpha>0$, resp. $\alpha=0$, we have 
\begin{align*}
\widetilde S_0&=(q-1)\left(\sum_{j=0}^{j_1-1}\frac{1}{(q^{N_js+\nu_j}-1)(q^{N_{j+1}s+\nu_{j+1}}-1)}+\sum_{j=1}^{j_1-1}\frac{1}{q^{N_js+\nu_j}-1}\right),\\
\text{resp.}\quad \widetilde S_0&=\sum_{j=1}^{j_1-1}\frac{q-1}{(q^{N_js+\nu_j}-1)(q^{N_{j+1}s+\nu_{j+1}}-1)}+\frac{q}{q^{N_1s+\nu_1}-1}+\sum_{j=2}^{j_1-1}\frac{q-1}{q^{N_js+\nu_j}-1}.
\end{align*}
For $\beta>0$, resp. $\beta=0$, we have
\begin{align*}
\widetilde S_k&=(q-1)\left(\sum_{j=j_k}^m\frac{1}{(q^{N_js+\nu_j}-1)(q^{N_{j+1}s+\nu_{j+1}}-1)}+\sum_{j=j_k+1}^m\frac{1}{q^{N_js+\nu_j}-1}\right),\\
\text{resp.}\quad \widetilde S_k&=\sum_{j=j_k}^{m-1}\frac{p-1}{(q^{N_js+\nu_j}-1)(q^{N_{j+1}s+\nu_{j+1}}-1)}+\frac{q}{q^{N_ms+\nu_m}-1}+\sum_{j=j_k+1}^{m-1}\frac{q-1}{q^{N_js+\nu_j}-1}.
\end{align*}

In the case $\alpha=\beta=0$, we claim that 
$$\widetilde S_0=\frac{\mathcal Q_1(q^s)}{q^{N(P_1)s+\nu(P_1)}-1}$$
and that
$$\widetilde S_k=\frac{\mathcal Q_2(q^s)}{q^{N(P_k)s+\nu(P_k)}-1}$$
for some 1-variable polynomials $\mathcal Q_1$ and $\mathcal Q_2$ over $\mathbb Q$. It suffices to check for $\widetilde S_0$ because, as we can see, the checking way also works for $\widetilde S_k$. Write
\begin{equation*}
\widetilde S_0=\frac{(q-1)\left(\sum\limits_{j=1}^{j_1-1}\prod\limits_{\begin{smallmatrix}1\leq t\leq j_1\\ t\not=j,j+1\end{smallmatrix}}\big(q^{N_ts+\nu_t}-1\big)+\sum\limits_{j=2}^{j_1-1}\prod\limits_{\begin{smallmatrix}1\leq t\leq j_1 \\ t\not=j\end{smallmatrix}}\big(q^{N_ts+\nu_t}-1\big)\right)+q\prod\limits_{t=2}^{j_1}\big(q^{N_ts+\nu_t}-1\big)}{\prod\limits_{t=1}^{j_1}\big(q^{N_ts+\nu_t}-1\big)}.
\end{equation*}
The sum of the terms appearing in the numerator of $\widetilde S_0$ not divisible by $q^{N_1s+\nu_1}-1$ is 
\begin{align*}
(q-1)\prod_{t=3}^{j_1}\big(q^{N_ts+\nu_t}-1\big)+q\prod_{t=2}^{j_1}\big(q^{N_ts+\nu_t}-1\big)=\prod_{t=3}^{j_1}\big(q^{N_ts+\nu_t}-1\big)\big(q^{d_2N_1s+d_2\nu_1}-1\big).
\end{align*}
It follows that the numerator of $\widetilde S_0$ is divisible by $q^{N(T_1)s+\nu(T_1)}-1$. For $2\leq j<j_1$, the sum of the terms appearing in the numerator of $\widetilde S_0$ not divisible by $q^{N_js+\nu_j}-1$ is 
\begin{gather*}
(q-1)\left(\prod_{\begin{smallmatrix}1\leq t\leq j_1\\ t\not=j-1,j\end{smallmatrix}}\!\big(q^{N_ts+\nu_t}-1\big)+\prod_{\begin{smallmatrix}1\leq t\leq j_1\\ t\not=j,j+1\end{smallmatrix}}\big(q^{N_ts+\nu_t}-1\big)+\prod_{\begin{smallmatrix}1\leq t\leq j_1\\ t\not=j\end{smallmatrix}}\big(q^{N_ts+\nu_t}-1\big)\!\right)\\
=(q-1)\prod_{\begin{smallmatrix}1\leq t\leq j_1\\ t\not=j-1,j,j+1\end{smallmatrix}}\big(q^{N_ts+\nu_t}-1\big)\big(q^{(N_{j-1}+N_{j+1})s+(\nu_{j-1}+\nu_{j+1})}-1\big). 
\end{gather*}
By Lemma \ref{lem37}, $q^{(N_{j-1}+N_{j+1})s+(\nu_{j-1}+\nu_{j+1})}-1$ is divisible by $q^{N_js+\nu_j}-1$. Also by Lemma \ref{lem37}, all the polynomials $q^{N_js+\nu_j}-1$ with $1\leq j<j_1$ are pairwise coprime. This follows the claim for $\widetilde S_0$.

Let us now compute $S_i$ for $1\leq i<k$. Writing $S_i$ as a fraction with denominator 
$$\prod_{j_i\leq t\leq j_{i+1}}\big(q^{N_ts+\nu_t}-1\big),$$ 
the sum of the terms in its numerator that are not divisible by $q^{N_js+\nu_j}-1$ for $j_i<j<j_{i+1}$ is 
\begin{gather*}
\prod_{\begin{smallmatrix}j_i\leq t\leq j_{i+1}\\ t\not=j-1,j\end{smallmatrix}}\big(q^{N_ts+\nu_t}-1\big)+\prod_{\begin{smallmatrix}j_i\leq t\leq j_{i+1}\\ t\not=j,j+1\end{smallmatrix}}\big(q^{N_ts+\nu_t}-1\big)+\prod_{\begin{smallmatrix}j_i\leq t\leq j_{i+1}\\ t\not=j\end{smallmatrix}}\big(q^{N_ts+\nu_t}-1\big)\\
=\prod_{\begin{smallmatrix}j_i\leq t\leq j_{i+1}\\ t\not=j-1,j,j+1\end{smallmatrix}}\big(q^{N_ts+\nu_t}-1\big)\big(q^{(N_{j-1}+N_{j+1})s+(\nu_{j-1}+\nu_{j+1})}-1\big),
\end{gather*}
which is divisible by $q^{N(T_j)s+\nu(T_j)}-1$ due to Lemma \ref{lem37}. If $\sum_{t=1}^ir_tb_t\not=\sum_{t=i+1}^kr_ta_t$, it follows from Lemma \ref{lem37} that all the polynomials $q^{N_js+\nu_j}-1$ with $j_i<j<j_{i+1}$ are pairwise coprime, hence there exists a $1$-variable polynomial $\mathcal P_i$ over $\mathbb Q$ such that
\begin{gather*}
\frac{(q-1)^2}{q^2}S_i=\frac{\mathcal P_i(q^s)}{\big(q^{N(P_i)s+\nu(P_i)}-1\big)\big(q^{N(P_{i+1})s+\nu(P_{i+1})}-1\big)}=Z_i.
\end{gather*}
For the case where $\sum_{t=1}^ir_tb_t=\sum_{t=i+1}^kr_ta_t$, the previous assertion still holds true, but it is proved in another way performed in Lemma \ref{lem2.4}. 

Finally, remark that computing $\widetilde S_0$ and $\widetilde S_k$ when $\alpha>0$ and $\beta>0$ can use the same arguments as computing $S_i$ for $1\leq i<k$, but it is easier because the situation as in Lemma \ref{lem2.4} does not happen. More precisely, there exist $1$-variable polynomials $\mathcal Q_1$ and $\mathcal Q_2$ over $\mathbb Q$ such that
$$
\widetilde S_0=\frac{\mathcal Q_1(q^s)}{(q^{\alpha s+1}-1)(q^{N(P_1)s+\nu(P_1)}-1)}
$$
and
$$
\widetilde S_k=\frac{\mathcal Q_2(q^s)}{(q^{\beta s+1}-1)(q^{N(P_k)s+\nu(P_k)}-1)}.
$$
The current expressions of $\widetilde S_0$ and $\widetilde S_k$ also work for $\alpha=\beta=0$ because in this case both $q^{\alpha s+1}-1$ and $q^{\beta s+1}-1$ will reduce to $q-1$. The proposition is completely proved.
\end{proof}

\begin{lemma}\label{lem37}
For $j_i\leq j, \ell <j_{i+1}$, $j\not=\ell$, 
$$\frac{\nu_j}{N_j}=\frac{\nu_{\ell}}{N_{\ell}}\quad \text{if and only if}\quad \sum_{t=1}^ir_tb_t=\sum_{t=i+1}^kr_ta_t.$$ 
In particular, for $j\not=\ell$ satisfying either $1\leq j, \ell <j_1$ or $j_k\leq j, \ell \leq m$, we have $\nu_j/N_j\not=\nu_{\ell}/N_{\ell}$. Moreover, for $P_i< T_j< P_{i+1}$, $0\leq i\leq k$, it admits identifications of natural numbers 
$$\frac{N_{j-1}+N_{j+1}}{N_j}=\frac{\nu_{j-1}+\nu_{j+1}}{\nu_j}=\frac{c_{j-1}+c_{j+1}}{c_j}=\frac{d_{j-1}+d_{j+1}}{d_j}.$$
\end{lemma}

\begin{proof}
This is an immediate consequence of Lemma \ref{local-form}.	The number $(d_{j-1}+d_{j+1})/d_j$ is a natural number because it is equal to $\det(T_{j-1},T_{j+1})$. 
\end{proof}

\begin{lemma}\label{lem2.4}
With the above notation and hypotheses, if $\sum_{t=1}^ir_tb_t=\sum_{t=i+1}^kr_ta_t$, then there exists a $1$-variable polynomial $\mathcal P_i^{\circ}$ over $\mathbb Q$ such that
\begin{gather*}
S_i=\frac{\mathcal P_i^{\circ}(q^s)}{\big(q^{N(P_i)s+\nu(P_i)}-1\big)\big(q^{N(P_{i+1})s+\nu(P_{i+1})}-1\big)}.
\end{gather*}

\end{lemma}

\begin{proof}
Put 
$$\delta := \sum_{t = 1}^i r_tb_t = \sum_{t = i+1}^n r_ta_t,$$ 
then we have 
$$N_j = c_j\sum_{t = 1}^i r_tb_t + d_j\sum_{t = i+1}^n r_ta_t = \delta(c_j + d_j)=\delta\nu_j.$$ 
Putting $T = q^{\delta s + 1}$, the sum $S_i$ defined in (\ref{sum-Si}) can be written as 
\begin{align*}
S_i = \sum_{j = j_i}^{j_{i+1} - 1}\frac{1}{(T^{\nu_j} - 1)(T^{\nu_{j+1}} - 1)} + \sum_{j = j_i + 1}^{j_{i+1} - 1}\frac{1}{(T^{\nu_j} - 1)}.
\end{align*}
For simplicity, we assume that $j_i = 1$ and $j_{i+1} = n$. Then we have
\begin{align*}
S_i = \dfrac{\sum\limits_{j = 1}^{n - 1}\left(\prod\limits_{1\leq k\leq n,k\ne j,j+1} (T^{\nu_k} - 1)\right) + \sum\limits_{j = 2}^{n - 1}\left(\prod\limits_{1\leq k\leq n,k\ne j} (T^{\nu_k} - 1)\right)}{(T^{\nu_1} - 1)(T^{\nu_n} - 1)\prod\limits_{j = 2}^{n - 1}(T^{\nu_j} - 1)}.
\end{align*}

It is a fact that 
$$T^{\nu_j} - 1 = \prod_{d|\nu_j}\varphi_d(T),$$ 
where $\varphi_d(T)$ for $d|\nu_j$ are cyclotomic polynomials and they are pairwise coprime. Then we have 
$$\prod\limits_{j = 2}^{n - 1}(T^{\nu_j} - 1) = \prod_{d \in \mathbb N}\left(\varphi_d(T)\right)^{\alpha_d},$$ 
where $\alpha_d$ is the number of elements divisible by $d$ of the set $\mathcal N:=\{\nu_2,\dots,\nu_{n-1}\}$. Take any $d$ with $\alpha_d=s>0$. To complete the proof it suffices to check that the numerator of $S_i$ is divisible by $\left(\varphi_d(T)\right)^{\alpha_d}$. Consider the subset $\{\nu_{i_1},\dots,\nu_{i_s}\}\subseteq \mathcal N$ consisting of all the elements divisible by $d$. Assume $i_1 < \cdots < i_s$. Consider the terms in the numerator of $S_i$. For any $j\in \{1,\dots,n-1\}\setminus \{i_l-1, i_l\mid 1\leq l\leq s\}$ we have $\{\nu_{i_1},\dots,\nu_{i_s}\}\subseteq \mathcal N\setminus\{j,j+1\}$, thus
$\prod\limits_{1\leq k\leq n,k\ne j,j+1} (T^{\nu_k} - 1)$ is divisible by $\prod_{l = 1}^s(T^{\nu_{i_l}} - 1)$. Since the latter is divisible by $\varphi_d(T)^{\alpha_d}$, it follows that $\prod\limits_{1\leq k\leq n,k\ne j,j+1}(T^{\nu_k} - 1)$ is divisible by $\varphi_d(T)^{\alpha_d}$. For $j\in \{2,\dots,n-1\}\setminus \{i_1,\dots,i_s\}$, the same argument shows that $\prod\limits_{1\leq k\leq n,k\ne j}(T^{\nu_k} - 1)$ is divisible by $\varphi_d(T)^{\alpha_d}$. The sum $R$ of the rest terms (i.e. the sum of all the terms not mentioned previously) in the numerator of $S_i$ is computed as follows
\begin{align}
R&=\sum_{l = 1}^s \left( \prod\limits_{\begin{smallmatrix}1\leq k\leq n\\ k\not=i_l-1,i_l\end{smallmatrix}}(T^{\nu_k} - 1) + \prod\limits_{\begin{smallmatrix}1\leq k\leq n\\ k\not=i_l,i_l+1\end{smallmatrix}}(T^{\nu_k} - 1)\right) + \sum_{l = 1}^s \left(\prod\limits_{\begin{smallmatrix}1\leq k\leq n\\ k\not=i_l\end{smallmatrix}}(T^{\nu_k} - 1)\right)\notag\\
&=\sum_{l = 1}^s \left(\prod\limits_{\begin{smallmatrix}1\leq k\leq n\\ k\not=i_l-1,i_l,i_l+1\end{smallmatrix}}(T^{\nu_k} - 1)\right)(T^{\nu_{i_l - 1} + \nu_{i_l + 1}} - 1 ).\label{smallidentity}
\end{align}
Since $\gcd(\nu_j,\nu_{j+1}) = 1$ for all $1\leq j\leq n-1$, we have $|i_l - i_{l+1}| \ge 2$ for all $1\leq l\leq s-1$. Therefore $\prod\limits_{1\leq k\leq n, k\ne i_l - 1,i_l,i_l + 1}(T^{\nu_k} - 1)$ is divisible by $\prod\limits_{1\leq t\leq s,t\ne l}(T^{\nu_{i_t}} - 1)$. Moreover, $T^{\nu_{i_l - 1} + \nu_{i_l + 1}} - 1$ is divisible by $T^{\nu_{i_l}} - 1$, so it turns out that each term in (\ref{smallidentity}) is divisible by $\prod_{t = 1}^s(T^{\nu_{i_t}} - 1)$. Thus, the sum $R$ is divisible by $\prod_{t = 1}^s(T^{\nu_{i_t}} - 1)=\varphi_d(T)^{\alpha_d}$.

The above arguments have shown that the numerator of $S_i$ is divisible by $\varphi_d(T)^{\alpha_d}$, hence the lemma is proved.
\end{proof}

\subsection{Some explicit computations of $Z(f;s)$ and $Z_0(f;s)$ for special polynomials}
Let $f(x,y)$ be a polynomial in $\mathcal O_K[x,y]$ such that $f(\mathbf 0)=0$ and $f(x,y)$ is Newton-nondegenerate over $K$. Assume that $f(x,y)$ admits the decomposition 
$$f(x,y)=cx^{\alpha}y^{\beta}\prod_{i=1}^k\prod_{\ell=1}^{r_i}f_{i\ell}(x,y)$$ 
in $K[[x,y]]$ with $c\in \mathcal O_K\setminus \mathcal M_K$ and $f_{i\ell}(x,y)$ having the initial expansion 
\begin{equation*} f_{i\ell}(x,y)=y^{a_i}+\xi_{i\ell}x^{b_i}+\text{(higher terms)},
\end{equation*}
in $K[[x,y]]$ with integers $a_i\geq 2$, $b_i\geq 2$, $(a_i,b_i)=1$, $\xi_{i\ell}\not=0$, $\xi_{i\ell}\not=\xi_{i\ell'}$ whenever $\ell\not=\ell'$. These conditions are a bit excessive but they are necessary to get a result analogous as the one in Proposition \ref{zeta-nondegenerate} for $Z(f;s)$. Remark that, besides requiring information about irreducible exceptional divisors, their intersections with the strict transforms, the description of $Z(f;s)$ also requires to know the whole strict transforms, while that of $Z_0(f;s)$ does not.

\begin{proposition}\label{formula-nondegenerate-polynomial}
Let $f(x,y)$ be as previous in this subsection and $\Phi$ as in Lemma \ref{strict-transform}. Assume that Condition \ref{assumption} is satisfied. Then $Z(f;s)$ is equal to 
$$
\frac{q^2-\#V(\widetilde{f})}{q^2}+\frac{q-1}{q^2}\sum_{j=1}^m\frac{\#\widetilde{E}(T_j)^{\circ}}{q^{N_js+\nu_j}-1}
$$
plus
$$\frac{(q-1)^2}{q^2}\left(\sum_{j=0}^m\frac{1}{(q^{N_js+\nu_j}-1)(q^{N_{j+1}s+\nu_{j+1}}-1)}+\sum_{i=1}^k\frac{r_i}{(q^{N(P_i)s+\nu(P_i)}-1)(q^{s+1}-1)}\right),
$$
with the convention for the second sum that	the term corresponding to $j=0$ (resp. $j=m$) is $0$ if $\alpha=0$ (resp. $\beta=0$).
\end{proposition}

\begin{proof}
Arguments are the same as those in the proof of Proposition \ref{main-formula-nondegenerate} with taking into account all the points of the strict transform of $\widetilde{\Phi}$.	
\end{proof}

\begin{proposition}
Let $f(x,y)$ be as previous and $\Phi$ as in Lemma \ref{strict-transform}. Assume that Condition \ref{assumption} is satisfied. Then $Z(f;s)$ is equal to
$$\frac{q^2-\#V(\widetilde{f})}{q^2}+\frac{\mathcal P(q^s)}{(q^{\alpha s+1}-1)(q^{\beta s+1}-1)(q^{N(P_1)s+\nu(P_1)}-1)(q^{N(P_k)s+\nu(P_k)}-1)}+ \sum_{i=1}^{k-1}Z_i +\widetilde Z,$$ 
where, for $1\leq i\leq k$, 
$$Z_i=\frac{\mathcal P_i(q^s)}{\big(q^{N(P_i)s+\nu(P_i)}-1\big)\big(q^{N(P_{i+1})s+\nu(P_{i+1})}-1\big)},$$
with $\mathcal P$ and $\mathcal P_i$, $1\leq i< k$, being univariate $\mathbb Q$-polynomials, and  
\begin{gather*}
\widetilde Z=\frac{q-1}{q^2}\sum_{i=1}^k\left(\frac{q-r_i-1}{q^{N(P_i)s+\nu(P_i)}-1}+\frac{(q-1)r_iq^{N(P_i)s+\nu(P_i)}}{\big(q^{N(P_i)s+\nu(P_i)}-1\big)\big(q^{s+1}-1\big)}\right). 
\end{gather*}
In other words, there is a univariate $\mathbb Q$-polynomials $\mathcal Q$ such that
$$Z(f;s)=\frac{\mathcal Q(q^s)}{(q^{s+1}-1)(q^{\alpha s+1}-1)(q^{\beta s+1}-1)\prod_{1\leq i\leq k}(q^{N(P_i)s+\nu(P_i)}-1)}.$$
\end{proposition}

\begin{proof}
This is a corollary of Propositions \ref{formula-nondegenerate-polynomial} and \ref{zeta-nondegenerate}, in which the latter contributes a proof method.	
\end{proof}


\begin{example}[$A_{2n}$- and $D_{2n+3}$-singularities]
Let $f(x,y) = y^2 - x^{2n+1}$. We consider the toric modification $\Phi$ admissible for $f$ that corresponds to the regular simplicial cone subdivision $T_1 < \cdots < T_{n+2}$, with $T_j = (1,j)^t$ for $1 \le j \le n$, $T_{n+1} = P_1 = (2,2n+1)^t$, and $T_{n+2} = (1,n+1)^t$. In this case, $K$ can be any finite extension of $\mathbb Q_p$ provided $q>2$, $\Phi$ is a resolution of $(f,\mathbf 0)$, and $\widetilde{\Phi}$ is a resolution of $(\widetilde{f},\widetilde{\mathbf 0})$. We have $(N_j, \nu_j) = (2j,j+1)$ for $1 \le j \le n$, $(N_{n+1},\nu_{n+1}) = (4n+2,2n+3)$, and $(N_{n+2},\nu_{n+2}) = (2n+1,n+2)$. We also have $\#V(\widetilde{f})=q$, $\#\widetilde{E}(T_j)^{\circ}=q-1$ for $2 \le j \le n$, $\#\widetilde{E}(T_1)^{\circ}=\#\widetilde{E}(T_{n+2})^{\circ}=q$, and $\#\widetilde{E}(T_{n+1})^{\circ}=q-2$. Then, by a direct computation, after Proposition \ref{formula-nondegenerate-polynomial}, we get that $Z(y^2 - x^{2n+1};s)$ is $q-1$ times of
\begin{align*}
\frac{(q^{s+1}-1)(q^{2ns +n+1}+1)\sum_{j=1}^n q^{2js+j-1} + q^{(2n+1)s+n}(q^{(2n+1)s+n+1}(q^{s+2} - 1) + q^{s+1} - 1)}{(q^{(4n+2)s + 2n+3} - 1)(q^{s+1} - 1)}.
\end{align*}

   A similar consideration for the plane singularity of type $D_{2n+3}$, that is, $f(x,y) = xy^2 - x^{2n+2} = x(y^2 - x^{2n+1})$ (in this case, one has $\alpha = 1$ and $\# V(\widetilde{f}) = 2q-1$), yields
   \begin{align*}
       \widetilde{S}_0 & = \frac{(q-1)q^{3s+2}(q^{(4n+1)s+2n+1} + (q^{(2n+1)s+n+1} + 1) \sum_{j=0}^{n-1}q^{2js + j})}{(q^{s+1} - 1)(q^{(4n+4)s+2n+3}-1)}, \\
       \widetilde{S}_1 & = \frac{q^{(2n+2)s + n+2} + 1}{q^{(4n+4)s+2n+3}-1},
   \end{align*}
   and thus
   \begin{align*}
       Z(f;s) & = \frac{(q-1)^2}{q^2} + \frac{q-1}{q^2}(\widetilde{S}_0 + \widetilde{S}_1) + \frac{q-1}{q^2}\left(\tfrac{q-2}{q^{(4n+4)s+2n+3}-1} + \tfrac{(q-1)q^{(4n+4)s+2n+3}}{(q^{s+1} - 1)(q^{(4n+4)s+2n+3}-1)}\right) \\
       & = \tfrac{(q-1)(q^{(2n+2)s+n}(q^{s+1}-1) + (q-1)q^{(4n+4)s+2n+1}(q^{s+1}+1) + (q-1)q^{3s}(q^{(2n+1)s+n+1} + 1)\sum_{j=0}^{n-1} q^{2js+j})}{(q^{s+1} - 1)(q^{(4n+4)s+2n+3}-1)}.
   \end{align*}
\end{example}

\begin{example}[$A_{2n-1}$- and $D_{2n+2}$-singularities]
    In this example, we assume $p \neq 2$. Let us consider the plane singularity of type $A_{2n-1}$ (where $n \ge 2$), that is, $f(x,y) = y^2 - x^{2n} = (y-x^n)(y+x^n)$. Then, $\alpha = \beta = 0$, $(a,b) = (1,n)$, and $r = 2$. Although we do not have $a \ge 2$, we can still apply the deploy the method from the proof of Proposition \ref{zeta-nondegenerate}, provided that $P_1 < T_m$. For this reason, we consider the regular simplicial cone subdivision $T_1 < \cdots < T_{n+1}$, with $T_j = (1,j)^t$. We have $(N_j, \nu_j) = (2j,j+1)$ for $1 \le j \le n$, and $(N_{n+1},\nu_{n+1}) = (2n,n+2)$. It follows that
\begin{align*}
     \widetilde{S}_0 & = \sum_{j=1}^{n-1}\frac{q-1}{(q^{2js + j+1} - 1)(q^{(2j+2)s+j+2} - 1)} + \frac{q}{q^{2s+2} - 1} + \sum_{j=2}^{n-1} \frac{q-1}{q^{2js+j+1} - 1} = \frac{1 + \sum_{j=1}^{n-1}q^{2js+j+1}}{q^{2ns+n+1}-1}
\end{align*}
and
\begin{align*}
    \widetilde{S}_1 & = \frac{q-1}{(q^{2ns+n+1} - 1)(q^{2ns+n+2} - 1)} + \frac{q}{q^{2ns+n+2} - 1} = \frac{1}{q^{2ns+n+1} - 1}.
\end{align*}
Since $\# V(\tilde{f}) = 2q-1$, the zeta function of $f(x,y) = xy^2 - x^{2n+1}$ is
\begin{align*}
       Z(f;s) & = \frac{q-1}{q^2}\left(q-1 + \widetilde{S}_0 + \widetilde{S}_1 + \frac{q-3}{q^{2ns+n+1} - 1} + \frac{2(q-1)q^{2ns+n+1}}{(q^{2ns+n+1} - 1)(q^{s+1} - 1)}\right) \\
       & = \frac{(q-1)((q-1)q^{2ns+n-1}(q^{s+1}+1) + (q^{s+1}-1)\sum_{j=1}^{n-1}q^{2js+j-1})}{(q^{2ns+n+1} - 1)(q^{s+1} - 1)}.
   \end{align*}
As for the plane singularity of type $D_{2n+2}$, that is, $f(x,y) = xy^2 - x^{2n+1} = x(y-x^n)(y+x^n)$, one has $\alpha = 1$ and $\# V(\tilde{f}) = 3q-2$. Hence
\begin{align*}
    \widetilde{S}_0 & = (q-1)\sum_{j=0}^{n-1}\left(\frac{1}{(q^{(2j+1)s + j+1} - 1)(q^{(2j+3)s+j+2} - 1)} + \frac{1}{q^{(2j+1)s+j+1} - 1} \right) \\
    & = \frac{(q-1)\sum_{j=1}^nq^{(2j+1)s+j+1}}{(q^{(2n+1)s+n+1}-1)(q^{s+1}-1)},\\
    \widetilde{S}_1 & = \frac{q-1}{(q^{(2n+1)s+n+1} - 1)(q^{(2n+1)s+n+2} - 1)} + \frac{q}{q^{(2n+1)s+n+2} - 1} = \frac{1}{q^{(2n+1)s+n+1} - 1},
\end{align*}
and thus
\begin{align*}
       Z(f;s) & = \frac{q-1}{q^2}\left(q-2 + \widetilde{S}_0 + \widetilde{S}_1 + \frac{q-3}{q^{(2n+1)s+n+1} - 1} + \frac{2(q-1)q^{(2n+1)s+n+1}}{(q^{(2n+1)s+n+1} - 1)(q^{s+1} - 1)}\right) \\
       & = \frac{(q-1)(q^{(2n+1)s+n}(q^{s+1} - 2q^{s} + 1) + (q-1)\sum_{j=1}^n q^{(2j+1)s+j-1}}{(q^{(2n+1)s+n+1} - 1))(q^{s+1} - 1)}.
   \end{align*}
\end{example}

\begin{example}[$E$-singularity]
We consider the $E_6$-plane curve singularity $f(x,y)=y^3-x^4$. We have $\alpha=\beta=0$, $(a,b)=(3,4)$, and $r=1$. Take the regular simplicial subdivision $(1,1)^t< (3,4)^t <(2,3)^t <(1,2)^t$. One has $(N_1, \nu_1)=(3,2)$, $(N_2,\nu_2)=(12,7)$, $(N_3,\nu_3)=(8,5)$, and $(N_4, \nu_4)=(4,3)$. Thus, we have 
\begin{align*}
    \widetilde{S}_0 & = \frac{q-1}{(q^{3s+2}-1)(q^{12s+7}-1)} + \frac{q}{q^{3s+2}-1}= \frac{q^{9s+6}+q^{6s+4}+q^{3s+2}+1}{q^{12s+7}-1},\\
    \widetilde{S}_1 &= \frac{q-1}{(q^{12s+7}-1)(q^{8s+5}-1)} + \frac{q-1}{(q^{8s+5}-1)(q^{4s+3}-1)} + \frac{q-1}{q^{8s+5}-1} + \frac{q}{q^{4s+3}-1}\\ &= \frac{q^{8s+5}+q^{4s+3}+1}{q^{12s+7}-1},
\end{align*}
and hence (note that $\#V(\widetilde{f)}= q$)
\begin{align*}
    Z(f;s)&= \frac{q-1}{q^2}\left(q+ \widetilde{S}_0 + \widetilde{S}_1 +\frac{q-2}{q^{12s+7}-1} +\frac{(q-1)q^{12s+7}}{(q^{12s+7}-1)(q^{s+1}-1)} \right)\\ 
    &= \frac{(q-1)q^{3s}(q^{10s+7}-q^{9s+5}+q^{7s+5}-q^{5s+3}+q^{4s+3}-q^{3s+2}+q^{2s+2}-1)}{(q^{12s+7}-1)(q^{s+1}-1)}.
\end{align*}
Next, consider the $E_7$-plane curve singularity $f(x,y)=y^3 - x^3y =y(y^2-x^3)$. In this case, $\alpha=0$, $\beta=1$, $(a,b)=(2,3)$, and $r=1$. Now, let us take the regular simplicial subdivision $(1,1)^t < (2,3)^t <(1,2)^t$. One has $(N_1, \nu_1)=(3,2)$, $(N_2,\nu_2)=(9,5)$, and $(N_3,\nu_3)=(5,3)$. It follows that 
\begin{align*}
    \widetilde{S}_0 & = \frac{q-1}{(q^{3s+2}-1)(q^{9s+5}-1)} + \frac{q}{q^{3s+2}-1}= \frac{q^{6s+4}+q^{3s+2}+1}{q^{9s+5}-1},\\
    \widetilde{S}_1 &= (q-1)\left(\frac{1}{(q^{9s+5}-1)(q^{5s+3}-1)} + \frac{1}{(q^{5s+3}-1)(q^{s+1}-1)} + \frac{1}{q^{5s+3}-1} + \frac{1}{q^{s+1}-1}\right)\\ &= \frac{(q-1)q^{3s+5}(q^{4s+2}+1)}{(q^{9s+5}-1)(q^{s+1}-1)},
\end{align*}
and, since $\#V(\widetilde{f)}= 2q-1$, the zeta function of $f(x,y) = y^3 - x^3y$ is
\begin{align*}
    Z(f;s) & = \frac{q-1}{q^2}\left(q-1 + \widetilde{S}_0 + \widetilde{S}_1 + \frac{q-2}{q^{9s+5}} + \frac{(q-1)q^{9s+5}}{(q^{9s+5}-1)(q^{s+1}-1)}\right)\\
    &= \frac{{q}^{3s}(q-1)( {q}^{7s+5}-{q}^{7s+4}+{q}^{6s+4}-{q}^{6s+3}+{q}^{4s+3}-{q}^{3s+2}+{q}^{2s+2} - {q}^{2s+1}+q^{s+1}-1)}{({q}^{9s+5}-1)(q^{s+1}-1)}.
\end{align*}
Finally, we deal with the plane curve singularity of type $E_8$, that is, $f(x,y)=y^3-x^5$. Now, we have $\alpha=\beta=0$, $(a,b)=(3,5)$, and $r=1$. Take the regular simplicial subdivision $(1,1)^t < (2,3)^t <(3,5)^t <(1,2)^t$. One has $(N_1, \nu_1)=(3,2)$, $(N_2,\nu_2)=(9,5)$, $(N_3,\nu_3)=(15,8)$, and $(N_4, \nu_4)=(5,3)$. This gives 
\begin{align*}
    \widetilde{S}_0 & = \frac{q-1}{(q^{3s+2}-1)(q^{9s+5}-1)} + \frac{q-1}{(q^{9s+5}-1)(q^{15s+8}-1)} + \frac{q}{q^{3s+2}-1} + \frac{q-1}{q^{9s+5}-1}\\ 
    &= \frac{q^{10s+6}+q^{6s+4}+q^{5s+3}+1}{q^{15s+8}-1},\\
    \widetilde{S}_1 &= \frac{q-1}{(q^{15s+8}-1)(q^{5s+3}-1)} + \frac{q}{(q^{5s+3}-1)}= \frac{q(q^{10s+5}+q^{5s+2}+1)}{q^{15s+8}-1},
\end{align*}
and hence (note that $\#V(\widetilde{f)}= q$) the zeta function of $f(x,y) = y^3 - x^5$ is given by
\begin{align*}
    Z(f;s)&= \frac{q-1}{q^2}\left(q+ \widetilde{S}_0 + \widetilde{S}_1 +\frac{q-2}{q^{15s+8}-1} +\frac{(q-1)q^{15s+8}}{(q^{15s+8}-1)(q^{s+1}-1)} \right)\\ 
    &= \frac{(q-1) {q}^{3s}( {q}^{13s+8}-{q}^{12s+6}+{q}^{10s+6}-{q}^{9s+5}+{q}^{8s+5}-{T}^{6s+3}+{q}^{4s+3}-{q}^{2s+1}+q^{s+1}-1)}{({q}^{15s+8}-1)(q^{s+1}-1)}.
\end{align*}
\end{example}


\section{General plane curve singularities}\label{resolution-tree}

\subsection{Resolution by toric modifications}\label{resolution-graph}
The content of this section is a $p$-adic version of Subsections 3.1 and 3.2 in \cite{Thuong-Hung}. Let $f(x,y)$ be a polynomial in $\mathcal O_K[[x]][y]$ that defines a plane curve $C\subseteq K^2$ near $\mathbf 0$. Let us write   
\begin{equation}\label{f-initial}
f(x,y)=cx^{\alpha}y^{\beta}\prod_{i=1}^k\prod_{\ell=1}^{r_i}\prod_{\tau=1}^{s_{i\ell}}f_{i\ell \tau}(x,y),
\end{equation}
where $c\in \mathcal O_K\setminus \mathcal M_K$, $\alpha, \beta\in\mathbb N$, and all $f_{i\ell\tau}$ are monic, non-smooth and irreducible in $\overline{K}[[x]][y]$. For convenience, we put 
$$f_i=\prod_{\ell=1}^{r_i}f_{\i\ell}\quad\text{and}\quad f_{\i\ell}=\prod_{\tau=1}^{s_{\i\ell}}f_{i\ell\tau},$$ 
and put 
$$A_i=\sum_{\ell=1}^{r_i}A_{i\ell}\quad\text{and}\quad  A_{i\ell}=\sum_{\tau=1}^{s_{i\ell}}A_{i\ell \tau}$$ 
for $1\leq i\leq k$ and $1\leq \ell\leq r_i$. Similarly as in \cite{AO}, each $f_{i\ell\tau}$ admits an initial expansion in the following form
\begin{equation}\label{ff-initial}
f_{i\ell\tau}(x,y)=(y^{a_i}+\xi_{i\ell}x^{b_i})^{A_{i\ell\tau}}+\text{(higher terms)},
\end{equation}
where $\xi_{i\ell}\in \overline{K}$, $\xi_{i\ell}\not=\xi_{i\ell'}$ if $\ell\not=\ell'$, $a_i$, $b_i$, and $A_{i\ell\tau}$ are in $\mathbb N^*$ with $(a_i,b_i)$ coprime. Since all the components $f_{i\ell\tau}$ are non-smooth, we have $a_i\geq 2$ and $b_i\geq 2$ for $1\leq i\leq k$. 

As in the previous section, we pose  
\begin{condition-0}
For $1\leq i\leq k$ and $1\leq \ell\leq r_i$, $\xi_{i\ell}\in \mathcal O_K\setminus \mathcal M_K$ and $f_{i\ell}$ are monic in $K[[x]][y]$.
\end{condition-0}

We are going to recall briefly the inductive construction of a resolution graph $\mathbf G=\mathbf G_f$ of $f$ by toric modifications, following \cite{Thuong-Hung} (see also \cite{Thuong1, Thuong2}). As explained in \cite{Thuong-Hung}, toric modifications making an embedded resolution of the singularity $(f,\mathbf 0)$ give {\it bamboos} in $\mathbf G$, the connection of consecutive bamboos is encoded by the canonical choice of Tschirnhausen coordinates. 

The vectors $P_i=(a_i,b_i)^t$, $1\leq i\leq k$, with $P_1<\cdots < P_k$, correspond to the compact facets of the Newton polyhedron $\Gamma_f$ of $f$. Let $\Phi_0$ be a toric modification admissible for $f$ that corresponds to a regular simplicial cone subdivision $T_1<\cdots<T_m$. Let $T_0=(1,0)^t$, $T_{m+1}=(0,1)^t$. The first bamboo $\mathscr B_0$ of $\mathbf G$ is the subgraph consisting of vertices $E(T_0)$ (for $\alpha>0$ only), $E(T_1),\dots, E(T_m)$, $E(T_{m+1})$ (for $\beta>0$ only), and of edges $[E(T_j),E(T_{j+1})]$, $0\leq j\leq m$ ($j=0$ for $\alpha>0$ only, $j=m$ for $\beta>0$ only). 

Assume by induction that $\mathscr B$ is a bamboo of $\mathbf G$ that has been already constructed, it corresponds to a toric modification $\Phi_{\mathscr B}$, and equivalently, to a regular simplicial cone subdivision $T^{\mathscr B}_1 <\cdots < T^{\mathscr B}_{m^{\mathscr B}}$. Let $f_{\mathscr B}$ be the pullback of $f$ under the composition of all toric modifications from the first one until the predecessor of $\Phi_{\mathscr B}$ for which $\Phi_{\mathscr B}$ is admissible. We assume that $f_{\mathscr B}$ is in $\mathcal O_K[[x_{\mathscr B}, y_{\mathscr B}]]$ and has the following initial expansion in the corresponding Tschirnhausen coordinates $(x_{\mathscr B},y_{\mathscr B})$:
\begin{align}\label{fB}
f_{\mathscr B}(x_{\mathscr B},y_{\mathscr B})=c_{\mathscr B}x_{\mathscr B}^{\alpha^{\mathscr B}}y_{\mathscr B}^{\beta^{\mathscr B}}\prod_{i=1}^{k^{\mathscr B}}\prod_{\ell=1}^{r^{\mathscr B}_i}\prod_{\tau=1}^{s^{\mathscr B}_{i\ell}}f^{\mathscr B}_{i\ell\tau}(x_{\mathscr B},y_{\mathscr B}),
\end{align}
where $c_{\mathscr B}\in \mathcal O_K\setminus \mathcal M_K$, $\alpha^{\mathscr B}, \beta^{\mathscr B}\in \mathbb N$, and
\begin{align}\label{fBir}
f^{\mathscr B}_{i\ell\tau}(x_{\mathscr B},y_{\mathscr B})=(y_{\mathscr B}^{a^{\mathscr B}_i}+\xi^{\mathscr B}_{i\ell}x_{\mathscr B}^{b^{\mathscr B}_i})^{A^{\mathscr B}_{i\ell\tau}}+\text{(higher terms)}
\end{align}
are non-smooth and irreducible in $\overline{K}[[x_{\mathscr B}]][y_{\mathscr B}]$, with $\xi^{\mathscr B}_{i\ell}\not=0$ distinct. For convenience, we put 
$$f_i^{\mathscr B}=\prod_{\ell=1}^{r_i^{\mathscr B}}f_{\i\ell}^{\mathscr B}\ \ \text{and}\ \ f_{\i\ell}^{\mathscr B}=\prod_{\tau=1}^{s_{\i\ell}}f_{i\ell\tau}^{\mathscr B},$$ 
and put 
$$A_i^{\mathscr B}=\sum_{\ell=1}^{r_i^{\mathscr B}}A_{i\ell}^{\mathscr B}\ \ \text{and}\ \ A_{i\ell}^{\mathscr B}=\sum_{\tau=1}^{s_{i\ell}^{\mathscr B}}A_{i\ell \tau}^{\mathscr B},$$ 
for $1\leq i\leq k^{\mathscr B}$ and $1\leq \ell\leq r_i^{\mathscr B}$. For $\mathscr B=\mathscr B_0$, it is possible that $\beta^{\mathscr B}\geq 0$, while for $\mathscr B\not=\mathscr B_0$, it is a fact that  $\beta^{\mathscr B}=0$. By \cite[Section 4.3]{AO}, $a_i^{\mathscr B}\geq 2$, $b_i^{\mathscr B}\geq 2$, $(a_i^{\mathscr B},b_i^{\mathscr B})=1$ for $1\leq i\leq k^{\mathscr B}$. Put $P^{\mathscr B}_i=(a^{\mathscr B}_i,b^{\mathscr B}_i)^t$ for $1\leq i\leq k^{\mathscr B}$, with $P^{\mathscr B}_1<\cdots <P^{\mathscr B}_{k^{\mathscr B}}$. 

\begin{condition-b}
For $1\leq i\leq k^{\mathscr B}$ and $1\leq \ell\leq r_i^{\mathscr B}$, $\xi_{i\ell}^{\mathscr B}\in \mathcal O_K\setminus \mathcal M_K$ and $f_{i\ell}^{\beta}$ are monic in $K[[x]][y]$.
\end{condition-b}

By \cite[Subsection 4.3]{AO}, the $A^{\mathscr B}_{i\ell}$-th Tschirnhausen approximate polynomial of $f^{\mathscr B}_{i\ell}$ has the form 
$$h^{\mathscr B}_{i\ell}(x_{\mathscr B},y_{\mathscr B})=y_{\mathscr B}^{a^{\mathscr B}_i}+\xi^{\mathscr B}_{i\ell}x_{\mathscr B}^{b^{\mathscr B}_i}+\text{(higher terms)}.$$ 
Let $(x_j,y_j)=(x_{\mathscr B,j},y_{\mathscr B,j})$ be the local coordinates of the toric chart corresponding to the matrix $(T_j^{\mathscr B},T_{j+1}^{\mathscr B})$. Consider $T^{\mathscr B}_j=P^{\mathscr B}_{i_0}$ for some $i_0\in\{1,\dots,k^{\mathscr B}\}$. Similarly as in \cite{AO}, we put
\begin{equation*}\label{T-change}
\begin{cases}
u=x_j\\
v=(\Phi_{\mathscr B}^*h^{\mathscr B}_{i_0\ell})/x_j^{a^{\mathscr B}_{i_0}b^{\mathscr B}_{i_0}}
\end{cases}
\end{equation*}
for some $R$ in $\overline{K}[[x_j,y_j]]$, from which (see \cite[Lemma 3.1]{Thuong-Hung})
\begin{equation*}
\begin{cases}
x_j=u\\
y_j=-\xi_{i_0\ell} + (-\xi_{i_0\ell})^{-c^{\mathscr B}_{j+1}b^{\mathscr B}_{i_0}}v+R'(u,v),
\end{cases}
\end{equation*}
for some $R'$ in  $\overline{K}[[u,v]]$. This system local coordinates $(u,v)$ is a system of Tschirnhausen coordinates, and $(\Phi_{\mathscr B}^*f_{\mathscr B})(u,v)$ admits a decomposition of the same type as (\ref{fB})-(\ref{fBir}) The Newton polyhedron of $\Phi_{\mathscr B}^*f_{\mathscr B}$ in the coordinates $(u,v)$ again gives rise to an admissible toric modification, which allows to construct a bamboo $\mathscr B'$, the successor of $\mathscr B$ in $\mathbf G$ at $E(P^{\mathscr B}_{i_0})$ associated to the branch $\ell$; hence we write $(x_{\mathscr B'}, y_{\mathscr B'})$ for $(u,v)$. The vertices of $\mathscr B'$ are $E(T^{\mathscr B'}_j)$, $1\leq j\leq m^{\mathscr B'}$, ordered as same as their sub-indices, the edges of $\mathscr B'$ are $[E(T^{\mathscr B'}_j),E(T^{\mathscr B'}_{j+1})]$, $0\leq j\leq m^{\mathscr B'}-1$, where $E(T^{\mathscr B'}_0)$ with $T^{\mathscr B'}_0=(1,0)^t$ is identified with $E(P^{\mathscr B}_{i_0})$. We usually denote $P^{\mathscr B'}_{\root}:=P^{\mathscr B}_{i_0}$. The graph $\mathbf G$ is called a {\it toric resolution tree} of $(f,\mathbf 0)$.  

We now write either $N_j^{\mathscr B}$ or $N(T_j^{\mathscr B})$ (resp. either $\nu_j^{\mathscr B}$ or $\nu(T_j^{\mathscr B})$) for the multiplicity of $\Phi_{\mathscr B}^*f_{\mathscr B}$ (resp. the pullback of $dx\wedge dy$ under the composition of all the toric modifications up to $\Phi_{\mathscr B}$) on $E(T_j^{\mathscr B})$. Then for $\mathscr B=\mathscr B_0$, $N_0=\alpha$ and $N_{m+1}=\beta$. For $\mathscr B\not=\mathscr B_0$, it together with (\ref{fB}) gives $N_0^{\mathscr B}=\alpha^{\mathscr B}=N(P_{\root}^{\mathscr B})$, and $\beta^{\mathscr B}=0$ because we use the Tschirnhausen coordinates. The following lemmas are similar as in \cite[Subsection 3.2]{Thuong-Hung}.

\begin{lemma}\label{lem41}
\begin{itemize}
\item[(i)] For $\mathscr B=\mathscr B_0$, and $1\leq j\leq m$ with $P_i\leq T_j \leq P_{i+1}$, $0\leq i\leq k$, we have
\begin{align*}
N_j=c_j\alpha+c_j\sum_{t=1}^ib_tA_t+d_j\sum_{t=i+1}^ka_tA_t+d_j\beta, \quad A_t=\sum_{\ell=1}^{r_t}\sum_{\tau=1}^{r_{t\ell}}A_{t\ell\tau}.
\end{align*}

\item[(ii)] For $\mathscr B\not=\mathscr B_0$ and $1\leq j\leq m^{\mathscr B}$ with $P^{\mathscr B}_i\leq T^{\mathscr B}_j \leq P^{\mathscr B}_{i+1}$, $0\leq i\leq k^{\mathscr B}$, we have
\begin{align*}
N^{\mathscr B}_j=c^{\mathscr B}_jN(P^{\mathscr B}_{\root})+c^{\mathscr B}_j\sum_{t=1}^ib^{\mathscr B}_tA^{\mathscr B}_t+d^{\mathscr B}_j\sum_{t=i+1}^{k^{\mathscr B}}a^{\mathscr B}_tA^{\mathscr B}_t,\quad A^{\mathscr B}_t=\sum_{\ell=1}^{r^{\mathscr B}_t}\sum_{\tau=1}^{r^{\mathscr B}_{t\ell}}A^{\mathscr B}_{t\ell\tau}.
\end{align*}
\end{itemize}
\end{lemma}

\begin{lemma}\label{lem42}
\begin{itemize}
	\item[(i)] For $\mathscr{B}=\mathscr{B}_0$ and $1\leq j\leq m$, we have $\nu_j=c_j+d_j$. 
	
	\item[(ii)] For $\mathscr{B}\not=\mathscr{B}_0$ and $1\leq j\leq m^{\mathscr B}$, we have $\nu_j^\mathscr{B}=c_j^\mathscr{B}\nu(P^{\mathscr B}_{\root})+d_j^\mathscr{B}$. 
\end{itemize}
\end{lemma}

\subsection{The $p$-adic zeta function of a plane curve singularity}
Let $f(x,y)\in \mathcal O_K[[x]][y]$ be a plane curve singularity at $\mathbf 0$. As explained in the introduction, we choose to compute $Z_0(f;s)$ instead of $Z(f;s)$ to avoid using the whole information about the strict transforms of an embedded resolution of singularities of $f$ (knowing the whole information is impossible, in general, because $K$ is not an algebraically closed field). Instead, the strict transforms only contribute to $Z_0(f;s)$ their intersection with the exceptional divisors. We need to extend $K$ to a reasonable size to guarantee our computation to run (see the hypotheses of the main theorems). In what follows, we are going to use the toric resolution tree $\mathbf G$ of $(f,\mathbf 0)$ constructed in Section \ref{resolution-graph}. We consider each bamboo $\mathscr B$ of $\mathbf G$ to be a subgraph of $\mathbf G$ with the single edge connecting $E(T_1^{\mathscr B})$ to $E(P_{\root}^{\mathscr B})$ included, we denote by $\mathbf B$ the set of bamboos of $\mathbf G$. 

\begin{theorem}\label{zeta-general}
Let $f(x,y)$ be a polynomial in $\mathcal O_K[[x]][y]$ that defines a plane curve $C\subseteq K^2$ at $\mathbf 0$. Assume that $f(x,y)$ admits the decomposition (\ref{f-initial}), the initial expansions (\ref{ff-initial}), and that Condition-$\mathscr B_0$ is satisfied. Assume in addition that an embedded resolution of $f(x,y)$ with the corresponding graph $\mathbf G$ can be constructed as in Section \ref{resolution-graph} by toric modifications $\Phi_{\mathscr B}$ thanks to the Tschirnhausen approximate polynomials, and that Condition-$\mathscr B$ are satisfied for all the non-top bamboos $\mathscr B$ of $\mathbf G$. For a top bamboo $\mathscr B$, put
$$Z_{\mathscr B}(s)=\frac{(q-1)^2/q^2}{(q^{N(P_{\root}^{\mathscr B})s+\nu(P_{\root}^{\mathscr B})}-1)(q^{s+1}-1)};$$
for a non-top bamboo $\mathscr B$, put 
\begin{gather*}
Z_{\mathscr B}(s)=\frac{q-1}{q^2}\sum_{j=1}^{m^{\mathscr B}}\frac{\#\widetilde{E}(T_j^{\mathscr B})^{\circ}}{q^{N_j^{\mathscr B}s+\nu_j^{\mathscr B}}-1}+\sum_{j=0}^{m^{\mathscr B}}\frac{(q-1)^2/q^2}{(q^{N_j^{\mathscr B}s+\nu_j^{\mathscr B}}-1)(q^{N_{j+1}^{\mathscr B}s+\nu_{j+1}^{\mathscr B}}-1)}\\
\qquad +\frac{(q-1)^2/q^2}{(q^{N(P_{\root}^{\mathscr B})s+\nu(P_{\root}^{\mathscr B})}-1)(q^{N_1^{\mathscr B}s+\nu_1^{\mathscr B}}-1)},
\end{gather*}
with convention for the second sum that	the term corresponding to $j=0$ (resp. $j=m^{\mathscr B}$) is $0$ if $\alpha^{\mathscr B}=0$ (resp. $\beta^{\mathscr B}=0$). Here $N_j^{\mathscr B}$, $\nu_j^{\mathscr B}$, $N(P_{\root}^{\mathscr B})$ and $\nu(P_{\root}^{\mathscr B})$ are given in Lemmas \ref{lem41} and \ref{lem42}. Then
$$Z_0(f;s)=\sum_{\mathscr B\in \mathbf B}Z_{\mathscr B}(s).$$
\end{theorem}

\begin{proof}
This is a direct corollary of Proposition \ref{main-formula} and the construction of $\mathbf G$. Namely, we apply the proposition as many times as the number of non-top bamboos of $\mathbf G$. 
\end{proof}

\begin{theorem}\label{zeta-truepoles}
Assume that the hypotheses of Theorem \ref{zeta-general} are satisfied. Let $\mathbf B^{\mathrm{nt}}$ denote the set of all the non-top bamboos of $\mathbf G$. Then there exist a unitvariate $\mathbb Q$-polynomial $\mathcal Q$ such that
$$Z_0(f;s)=\frac{\mathcal Q(q^s)}{(q^{s+1}-1)(q^{\alpha s+1}-1)(q^{\beta s+1}-1)\prod_{\mathscr B\in \mathbf B^{\mathrm{nt}}}\prod_{1\leq i\leq k}(q^{N(P_i^{\mathscr B})s+\nu(P_i^{\mathscr B})}-1)}.$$

%
\end{theorem}

\begin{proof}
We apply Theorem \ref{zeta-general} and use the same method as in the proof of Proposition \ref{zeta-nondegenerate}.
\end{proof}

\begin{remark}
It follows from Theorem \ref{zeta-truepoles}, any pole of the $p$-adic zeta function $Z_0(f;s)$ is either 
$$-1+\frac{2l\pi\sqrt{-1}}{\ln q}$$ 
or has the form 
$$-\frac{\nu(P_i^\mathscr{B})}{N(P_i^\mathscr{B})}+\frac{2l\pi\sqrt{-1}}{(\ln q) N(P_i^\mathscr{B})}$$ 
for some $\mathscr{B}\in\mathbf B$, $1\leq i\leq k^\mathscr{B}$, and for every $l\in \mathbb Z$. Thus, similarly as in \cite[Theorem 4.2]{Thuong-Hung}, Theorem \ref{zeta-general} and \cite[Theorem 3]{AC1} are the ingredients to give an elementary proof of the $p$-adic monodromy conjecture. We can also deduce from Theorem \ref{zeta-general} that the $p$-adic zeta function $Z_0(f;s)$ is a topological invariant (compare with \cite[Theorem 4.1]{Thuong-Hung}).
\end{remark}


\begin{ack}
The second author thanks the Vietnam Institute for Advanced Study in Mathematics (VIASM) for warm hospitality during his visit. He would also like to acknowledge support from the ICTP through the Associates Programme (2020-2025). 
\end{ack}

\end{document}